\documentclass[a4paper,10pt]{article}

\usepackage{amsthm}
\usepackage{amsmath}
\usepackage{amsxtra}
\usepackage{amssymb}
\usepackage{thmtools}
\usepackage{hyperref}
\usepackage{mathrsfs}
\usepackage{undertilde}
\usepackage{turnstile}
\usepackage{enumitem}

\usepackage[retainorgcmds]{IEEEtrantools}

\declaretheoremstyle[bodyfont=\sl]{slanted}

\swapnumbers
\declaretheorem[name=Definition,style=definition,qed=$\dashv$,
numberwithin=section]{dfn}
\declaretheorem[name=Definition,style=definition,numbered=no,qed=$\dashv$]{dfn*}
\declaretheorem[name=Definition,style=definition,numbered=no]{dfnnoqed*}

\declaretheorem[name=Theorem,style=slanted,sibling=dfn]{tm}
\declaretheorem[name=Theorem,style=slanted,numbered=no]{tm*}
\declaretheorem[name=Lemma,style=slanted,sibling=dfn]{lem}
\declaretheorem[name=Corollary,style=slanted,sibling=dfn]{cor}
\declaretheorem[name=Corollary,style=slanted,numbered=no]{cor*}
\declaretheorem[name=Remark,style=definition,sibling=dfn]{rem}
\declaretheorem[name=Question,style=definition,sibling=dfn]{ques}

\declaretheorem[name=Fact,style=definition,sibling=dfn]{fact}

\swapnumbers
\declaretheoremstyle[headfont=\scshape]{claimstyle}
\declaretheorem[name=Claim,style=claimstyle]{clm}
\declaretheorem[name=Claim,style=claimstyle]{clmtwo}
\declaretheorem[name=Claim,style=claimstyle]{clmthree}
\declaretheorem[name=Claim,style=claimstyle]{clmfour}

\newcommand{\Lim}{\mathrm{Lim}}
\newcommand{\illfp}{\mathrm{illfp}}
\newcommand{\ind}{\mathrm{ind}}
\newcommand{\elmt}{\mathrm{elmt}}

\newcommand{\pvec}{\vec{p}}

\newcommand{\ins}{\trianglelefteq}
\newcommand{\iso}{\cong}

\newcommand{\RR}{\mathbb R}

\newcommand{\PP}{\mathbb P}

\newcommand{\sub}{\subseteq}
\newcommand{\cross}{\times}
\newcommand{\all}{\forall}

\newcommand{\inter}{\cap}
\renewcommand{\int}{\inter}

\newcommand{\om}{\omega}
\newcommand{\pow}{\mathcal{P}}
\newcommand{\OR}{\mathrm{OR}}

\newcommand{\Hull}{\mathrm{Hull}}

\newcommand{\cut}{\backslash}

\newcommand{\Ss}{\mathcal{S}}

\newcommand{\Ll}{\mathcal{L}}

\newcommand{\rg}{\mathrm{rg}}
\newcommand{\dom}{\mathrm{dom}}

\newcommand{\crit}{\mathrm{cr}}

\newcommand{\rest}{\!\upharpoonright\!}
\newcommand{\com}{\circ}

\newcommand{\Ult}{\mathrm{Ult}}

\newcommand{\sats}{\models}
\newcommand{\elem}{\preccurlyeq}
\newcommand{\J}{\mathcal{J}}

\newcommand{\AC}{\mathsf{AC}}
\newcommand{\DC}{\mathsf{DC}}

\newcommand{\HOD}{\mathrm{HOD}}

\newcommand{\ZFC}{\mathsf{ZFC}}
\newcommand{\ZF}{\mathsf{ZF}}

\newcommand{\id}{\mathrm{id}}

\newcommand{\forces}{\dststile{}{}}
\newcommand{\bfPi}{\utilde{\Pi}}
\newcommand{\bfSigma}{\utilde{\Sigma}}

\newcommand{\rSigma}{\mathrm{r}\Sigma}

\newcommand{\sig}{\mathrm{sig}}

\newcommand{\RL}{\mathsf{RL}}
\newcommand{\scot}{\mathrm{scot}}
\newcommand{\ZFR}{\mathrm{ZFR}}

\newcommand{\spt}{\mathrm{spt}}

\newcommand{\bfrSigma}{\utilde{\rSigma}}

\newcommand{\psub}{\subsetneq}

\newcommand{\lpole}{\left\lfloor}
\newcommand{\rpole}{\right\rfloor}
\newcommand{\univ}[1]{\lpole #1\rpole}
\newcommand{\tu}{\textup}
\newcommand{\NN}{\mathbb{N}}
\newcommand{\lex}{\mathrm{lex}}

\DeclareMathOperator{\Th}{Th}

\DeclareMathOperator{\cof}{cof}
\DeclareMathOperator{\wfp}{wfp}
\DeclareMathOperator{\rank}{rank}

\begin{document}

\title{Extenders under ZF and\\
constructibility of rank-to-rank embeddings}

\author{Farmer Schlutzenberg\footnote{Funded by the Deutsche 
Forschungsgemeinschaft (DFG, German Research
Foundation) under Germany's Excellence Strategy EXC 2044-390685587,
Mathematics M\"unster: Dynamics-Geometry-Structure.}\\
farmer.schlutzenberg@gmail.com}

\maketitle

\begin{abstract}
Assume ZF (without the Axiom 
of Choice). Let $j:V_{\varepsilon}\to V_\delta$ be a non-trivial 
$\in$-cofinal $\Sigma_1$-elementary 
embedding,
where $\varepsilon,\delta$ are limit ordinals.
We prove some 
restrictions
on the constructibility of $j$ from $V_\delta$,
mostly focusing on the case $\varepsilon=\delta$.
In particular, if $\varepsilon=\delta$ and $j\in L(V_\delta)$ then 
$\cof(\delta)=\om$.
However, assuming ZFC + I$_3$,
with the appropriate $\varepsilon=\delta$, one can force
to get such
$j\in L(V^{V[G]}_\delta)$.
Assuming Dependent Choice  and  $\cof(\delta)=\om$  (but not assuming 
$V=L(V_\delta)$), and $j:V_\delta\to V_\delta$ is  $\Sigma_1$-elementary, we 
show 
that the there ``perfectly many' such $j$, with none being ``isolated''.
Assuming
a proper class of weak Lowenheim-Skolem cardinals,
we also  give a first-order 
characterization
of critical points of embeddings $j:V\to M$ with $M$ 
transitive. The main results rely on a development of extenders
under ZF (which is most useful given such wLS cardinals).\end{abstract}

\section{Introduction}

Large cardinal axioms are typically exhibited by a class elementary 
embedding
\[ j:V\to M \]
where $V$ is the universe of all sets and $M\sub V$ some transitive inner model.
Stronger large cardinal notions are obtained by demanding more 
resemblance between $M$ and $V$, and by demanding that more of $j$ 
gets into $M$.  William Reinhardt, in \cite{reinhardt_diss} and 
\cite{reinhardt_remarks}, took this notion to its extreme by setting $M=V$;
that is, he considered embeddings of the form
\[ j:V\to V.\]
The critical point\footnote{The least ordinal $\kappa$ such that 
$j(\kappa)>\kappa$.} of such an embedding became known as a \emph{Reinhardt 
cardinal}.

But Kunen showed in \cite{kunen_no_R} that if $V\sats\ZFC$ then there is no 
such 
$j$, and  in fact no ordinal $\lambda$ and elementary 
embedding $j:V_{\lambda+2}\to V_{\lambda+2}$.\footnote{If there is
$j:V\to V$ then there is a proper class of ordinals $\alpha$
such that $j:V_\alpha\to V_\alpha$, and then $j:V_{\alpha+n}\to V_{\alpha+n}$ 
for each $n<\om$.} Given  this violation of $\AC$, most research in large 
cardinals
following Kunen's discovery has been focused below the 
$\AC$ threshold.
In this paper, we drop the $\AC$ assumption,
and investigate embeddings $j:V_\delta\to V_\delta$,
 assuming only $\ZF$ in $V$. (We mention it explicitly
 when we use some $\DC$.)
 Thus, Kunen's objections are not valid here, and
moreover, no one has succeeded in refuting 
$j:V_{\lambda+2}\to V_{\lambda+2}$, or even $j:V\to V$, in $\ZF$ alone,
except for in the sense we describe next.

The following theorem is 
Suzuki \cite[Theorem 2.1]{suzuki_no_def_j}:

\begin{fact}[Suzuki] Assume $\ZF$ and let $j\sub V$ be a class which is 
definable from parameters.\footnote{Here given a structure $M$
and $X\sub M$, we say that $X$ is \emph{definable from parameters over} $M$
iff there is $\pvec\in M^{<\om}$ and a formula $\varphi$
such that for all $x\in M$, we have $x\in X$ iff $M\sats\varphi(x,\pvec)$.} 
Then it is not the case that $j:V\to V$ is a $\Sigma_1$-elementary 
embedding.\footnote{Actually, because $V\sats\ZF$,
$\Sigma_1$-elementarity is equivalent to full elementarity in this 
case.}\end{fact}

So if we interpret ``class'' as requiring definability from parameters -- 
which is one standard interpretation -- then Suzuki's theorem actually
disproves the existence of Reinhardt cardinals from $\ZF$.
However, there are other interpretations of ``class'' which are more liberal,
and for which this question is still open. (And significantly, note that this 
theorem doesn't seem to say anything about the existence of an elementary 
$j:V_{\lambda+2}\to V_{\lambda+2}$.)

What can  ``classes'' be,
if not definable from parameters?
One way to consider this question,  is to shift perspective:
instead of considering $V$, let us consider some rank segment $V_\delta$ of $V$,
where $\delta$ 
is an ordinal. 
The sets $X\sub V_\delta$ such that $X$ is definable from parameters
over $V_\delta$, are exactly those sets $X\in L_1(V_\delta)$,
the first step in G\"odel's constructibility hierarchy above base set 
$V_\delta$.\footnote{That is, we start with $V_\delta$ as the first
set in the hierarchy, instead of $L_0=\emptyset$, and then proceed 
level-by-level as for $L$.} So by Suzuki's theorem,
if $V_\delta\sats\ZF$, then there is no $\Sigma_1$-elementary $j:V_\delta\to 
V_\delta$ with $j\in L_1(V_\delta)$. Viewed this way, we now have the obvious 
question, as to whether we can have such a $j\in L_\alpha(V_\delta)$, for 
larger $\alpha$. And the subsets of $V_\delta$ which are in $L(V_\delta)$, for 
example, provide another reasonable notion of ``class'' with respect to 
$V_\delta$.

So, wanting to examine these issues, 
this paper focuses mostly on the question of the constructibility of
 such $j$
 from $V_\delta$,
looking for generalizations of 
Suzuki's theorem, and examining the consequences of having such $j\in 
L(V_\delta)$.
We remark that the papers \cite{cumulative_periodicity}, 
\cite{goldberg_even_ordinals} develop some general analysis
of \emph{rank-to-rank} embeddings $j:V_\delta\to V_\delta$,
for arbitrary ordinals $\delta$, where we include all $\Sigma_1$-elementary
maps in this notion. But in this paper, we consider only the case 
that $\delta$ is a limit.\footnote{Most of the results
in this paper first appeared
in the author's informal notes \cite{reinhardt_non-definability}.
Those notes were broken into pieces according to theme,
and this paper is one of 
those pieces. 
There are also some further observations here not present in 
\cite{reinhardt_non-definability}, mainly in \S\ref{sec:admissible}.
We also corrected an error in the definition of \emph{extender}
given in \cite{reinhardt_non-definability}; see \S\ref{sec:ultrapowers_in_ZF}
of the present paper, and in particular Footnote \ref{ftn:rank-ext}.}
For convenience, we write $\mathscr{E}_m(V_\delta)$ for the set of all 
$\Sigma_m$-elementary
$j:V_\delta\to V_\delta$, and $\mathscr{E}(V_\delta)=\mathscr{E}_\om(V_\delta)$.

We also consider, to some extent, analogous constructibility questions
for $\in$-cofinal $\Sigma_1$-elementary embeddings $j:V_{\bar{\delta}}\to 
V_\delta$ where $\bar{\delta}<\delta$; of course these embeddings are
much lower in consistency strength. Here the picture regarding constructibility 
seems murkier.

The first basic result here is the following
(in the case that $\delta$ is inaccessible,
the result is due independently and earlier to Goldberg, who used other 
methods):

\begin{tm*}[\ref{tm:V=L(V_delta)_cof(delta)>om}]
  Assume $\ZF+V=L(V_\delta)$ where
 $\cof(\delta)>\om$. Then $\mathscr{E}_1(V_\delta)=\emptyset$; that is,
there is no $\Sigma_1$-elementary $j:V_\delta\to V_\delta$.
\end{tm*}

We also establish a restriction on $\in$-cofinal $\Sigma_1$-elementary
$j:V_{\bar{\delta}}\to V_\delta$ when $V=L(V_\delta)$ and $\cof(\delta)>\om$
and $\bar{\delta}<\delta$,
although our proof does not rule this situation out.
Considering the case  that $\cof(\delta)=\om$, we  establish:
\begin{tm*}[\ref{tm:no_j_in_adm}] Assume $\ZF$, $\delta\in\Lim$
with $\cof(\delta)=\om$, $\bar{\delta}\leq\delta$ and $j:V_{\bar{\delta}}\to 
V_\delta$
is $\in$-cofinal $\Sigma_1$-elementary.
Let $\kappa\in\OR$ be least with $L_\kappa(V_\delta)$ admissible.
Then $j\notin L_\kappa(V_\delta)$, and moreover,  $j$ is not 
$\bfSigma_1^{L_\kappa(V_\delta)}$,
and  not $\bfPi_1^{L_\kappa(V_\delta)}$.
\end{tm*}

However, assuming $\ZFC+I_3$, one can force ``there
is 
$\delta\in\Lim$
and an elementary $j:V_\delta\to V_\delta$
with $j\in L_{\kappa+1}(V_\delta)$ (with $\kappa$ as above)'';
see \ref{tm:j_in_L(V_delta)_under_V=HOD}.

The question of constructibility of rank-to-rank embeddings is also 
related to questions on their uniqueness. That is, consider an $I_3$ 
rank-to-rank embedding 
$j$
(so $j:V_\delta\to V_\delta$ where $\delta$ is a limit and 
$\delta=\kappa_\om(j)=\lim_{n<\om}\kappa_n$ where $\kappa_0=\crit(j)$ and 
$\kappa_{n+1}=j(\kappa_n)$). Each finite iterate $j_n$ is also such an 
embedding,
and $\lim_{n<\om}\crit(j_n)=\delta$. It follows easily that
for each $\alpha<\delta$, there are multiple elementary $k:V_\delta\to V_\delta$
such that $k\rest V_\alpha=j\rest V_\alpha$ (just consider $j_n(j)$ for some 
sufficiently large $n$). But what if instead, $j:V_{\lambda+\om}\to 
V_{\lambda+\om}$
is elementary where $\lambda=\kappa_\om(j)$? Can there be
$n<\om$ such that $j$ is the unique elementary $k:V_{\lambda+\om}\to
V_{\lambda+\om}$
extending $j\rest V_{\lambda+n}$? The $I_3$ argument clearly doesn't work here,
since an elementary $k:V_{\lambda+\om}\to V_{\lambda+\om}$
must have $\crit(k)<\lambda$. The answer is ``no'' under
$\DC$:

\begin{tm*}[\ref{tm:DC_or_force_embeddings}] Assume $\ZF$ and let $\delta$
be  a limit ordinal
where $\mathscr{E}(V_\delta)\neq\emptyset$,
and $m\in[1,\om]$.
 Then:
 \begin{enumerate}
  \item In a set-forcing extension of $V$, for each 
$V$-amenable\footnote{\emph{$V$-amenable}
means that $j\rest V_\alpha\in V$ for each $\alpha<\delta$.}
$j\in\mathscr{E}_m(V_\delta)$ and $\alpha<\delta$
 there is a $V$-amenable $k\in\mathscr{E}_m(V_\delta)$
 with $k\rest V_\alpha=j\rest V_\alpha$
 but $k\neq j$.

  \item If $\DC$ holds and $\cof(\delta)=\om$ then for each 
$j\in\mathscr{E}_m(V_\delta)$ and $\alpha<\delta$ there 
is 
 $k\in\mathscr{E}_m(V_\delta)$ with $k\rest V_\alpha=j\rest 
V_\alpha$
 but $k\neq j$.
 \end{enumerate}
\end{tm*}

More generally, given a $\Sigma_1$-elementary $j\in\mathscr{E}_1(V_\delta)$
with limit $\delta$,
one can ask whether
$j\in\HOD(V_\delta)$. (The $I_3$ example above does give an instance of this,
since $L(V_\delta)\sub\HOD(V_\delta)$.)

We also prove the following related result, which is another
kind of 
strengthening of Suzuki's theorem above, and also a strengthening of
\cite[Theorem 5.7 part 1***]{cumulative_periodicity}.
The theory $\RL$ (\emph{Rank Limit}, see Definition \ref{dfn:RL}) is a 
sub-theory of 
$\ZF$, which holds in  $V_\eta$ for limit $\eta$:

\begin{tm*}[\ref{tm:RL_no_def_j}]
 Assume $\RL$. Then there is no $\Sigma_1$-elementary $j:V\to V$ which is 
definable from parameters.
\end{tm*}

As a useful tool, and also for more general use elsewhere,
we develop the theory of extenders and ultrapowers under ZF (see 
\S\ref{sec:ultrapowers_in_ZF}). Of course
a key fact when considering ultrapowers is {\L}o\'{s}' theorem,
which is related to $\AC$. However, it turns out 
that
a proper class of weak L\"owenheim-Skolem (wLS) cardinals  
(due to Usuba \cite{usuba_ls}; see Definition \ref{dfn:wLS}) gives
enough choice to secure {\L}o\'{s}' theorem for the kinds of extenders over $V$
we will consider here. 
A proper class of wLS cardinals is not known to be inconsistent 
with choiceless large cardinals, and is in fact implied by a super 
Reinhardt; it also holds if  $\AC$ is forceable with a set-forcing
(see \cite{usuba_ls}).
Under this large cardinal assumption, we will prove the 
following generalization that measurability classifies critical points
under $\ZFC$; we say that an ordinal $\kappa$ is \emph{$V$-critical}
iff there is an elementary $j:V\to M$ with $\crit(j)=\kappa$:

\begin{tm*}[\ref{tm:crit_def}]
Assume $\ZF+$``there is a proper class of weak L\"owenheim-Skolem cardinals''.
Then the class of $V$-critical ordinals is definable without parameters.
\end{tm*}

\subsection{Background and terminology}\label{subsec:notation}

We now list some background and terminology, which
is pretty standard. Our basic background
theory throughout is $\ZF$
(in fact, many of the results require much less, but we leave that to the 
reader), with additional hypotheses stated
as required.
However, in \S\ref{sec:RL}, we work in the weaker theory $\RL$.
In the $\ZF$ (or weaker)
context, certain familiar definitions
from the $\ZFC$ context need to be modified appropriately.

 Work in $\ZF$. $\OR$ denotes the class of ordinals and $\Lim$ the class of 
limit ordinals. Let $\delta\in\Lim$.
The \emph{cofinality} of $\delta$, \emph{regularity}, \emph{singularity}
are defined as usual (in terms of cofinal functions between ordinals).
 We say  $\delta$ (or $V_\delta$) is \emph{inaccessible}
 iff there is no $(\gamma,f)$ such that $\gamma<\delta$
 and $f:V_\gamma\to\delta$ is cofinal in $\delta$.
  A \emph{norm} on a set $X$ is a surjective function $\pi:X\to\alpha$
 for some $\alpha\in\OR$. The associated \emph{prewellorder}
 on $X$ is the relation $R\sub X^2$ where $xRy$ iff $\pi(x)\leq\pi(y)$.
This can of course be inverted.
If $\delta\in\OR$ is regular but non-inaccessible,
then the \emph{Scott ordertype} $\scot(\delta)$ of $\delta$
is the set $P$ of all prewellorders of $V_{\alpha+1}$ in ordertype $\delta$,
where $\alpha$ is least admitting such. See 
\cite[5.2, 5.3***]{cumulative_periodicity} for some basic facts about this.
A partial function $f$ from (some subset of) $X$ to $Y$
is denoted $f:_{\mathrm{p}}X\to Y$.

  Let $M=(U,\in^M,A_1,\ldots,A_n)$ be a first-order structure with
 universe $U$ and $A\sub U$. (We normally abbreviate
 this by just writing $A\sub M$.)
 We say that $A$ is \emph{definable over }$M$\emph{ from parameters}
 if there is a first-order formula $\varphi\in\Ll_{\dot{A}_1,\ldots,\dot{A}_n}$
(with symbols $\in,=,\dot{A}_1,\ldots,\dot{A}_n$) and some $\vec{x}\in 
M^{<\om}$
 such that for all $y\in M$,
 we have $y\in A$ iff $M\sats\varphi(\vec{x},y)$.
 This is naturally refined by \emph{$\Sigma_n$-definable from parameters},
 if we require $\varphi$ to be $\Sigma_n$, and by \emph{from parameters in }$X$,
 if we restrict to $\vec{x}\in X^{<\om}$ (where $X\sub M$),
 and by \emph{from $\vec{x}$}, if we may use only $\vec{x}$.

For an extensional structure $M=(\univ{M},\in^M,=^M)$,
the \emph{wellfounded 
part} $\wfp(M)$
of $M$ is the class of all transitive isomorphs of elements of $M$.
That is, the class of all transitive sets $x$
such that $(x,{\in}\rest x,{=}\rest x)$
is isomorphic to $(y,\in^M\rest y,=^M\rest y)$
for some $y\in\univ{M}$. The \emph{illfounded part} $\illfp(M)$
of $M$ is $\univ{M}\cut\wfp(M)$. We have
$\wfp(M),\illfp(M)\in L(M)$, which results by ranking the elements
of $M$ as far as is possible. Note that $\illfp(M)$
is the largest $X\sub\univ{M}$ such that for every $x\in X$
there is $y\in X$ such that $y\in^Mx$. If $M$ models
enough set theory that it has a standard rank function,
but $M$ is illfounded, then 
$\OR^M$
(the collection of all $x\in\univ{M}$ such that $M\sats$``$x$ is an 
ordinal'')
is illfounded,
because if $x\in^M y$ then $\rank^M(x)\in^M\rank^M(y)$.

 Given a structure $M$ and $k\in[1,\om]$, $\mathscr{E}_k(M)$
denotes the set of all $\Sigma_k$-elementary $j:M\to M$,
and $\mathscr{E}(M)$ denotes $\mathscr{E}_\om(M)$.
Let $\delta\in\Lim$ and $j\in\mathscr{E}_1(V_\delta)$.
For  $C\sub V_\delta$, define
$j^+(C)=\bigcup_{\alpha<\delta}j(C\inter V_\alpha)$.
 Let $\kappa_0=\crit(j)$ and $\kappa_{n+1}=j(\kappa_n)$
and
$\kappa_\om(j)=\sup_{n<\om}\kappa_n$ (note 
$\kappa_\om(j)\leq\delta$).
We write $j_0=j$ and $j_{n+1}=(j_n)^+(j_n)$ for $n<\om$;
then $j_n\in\mathscr{E}(V_\delta)$ (see 
\cite[Theorem 5.6***]{cumulative_periodicity})
and $\kappa_n=\crit(j_n)$.

$\ZF_2$ denotes the two-sorted theory, with models of the form $M=(V,E,P)$,
where $(V,E)\sats\ZF$ and $P$ is a collection of ``subsets'' of $V$.
The elements of $V$ are the \emph{sets} of $M$ and the elements of $P$
the \emph{classes}. The axioms are like those of $\ZF$,
but include Separation for sets and Collection for sets
with respect to all formulas of the two-sorted language,
and also Separation for classes with respect to such formulas.
When we ``work in $\ZF_2$'', we mean that we work in
such a model $M$, and all talk of proper classes refers to elements of $P$.
Note that the theory is first-order; there can be countable models of $\ZF_2$.
One has  $(V_\kappa,\in,V_{\kappa+1})\sats\ZF_2$
iff $\kappa$ is inaccessible.

The \emph{language of set theory with predicate}
$\Ll_{\in,A}$ is the first order language with binary 
predicate symbol
$\in$ and predicate symbol $A$.
The theory \emph{$\ZF(A)$} is the theory in 
$\Ll_{\in,A}$
with all $\ZF$ axioms, allowing all 
formulas of $\Ll_{\in,A}$
in the Separation and Collection schemes (so $A$ represents a class).
$\ZFR$ ($\ZF+$ Reinhardt) is the theory $\ZF(j)+$``$j:V\to V$ is 
$\Sigma_1$-elementary''.
By \cite[Proposition 5.1]{kanamori}, $\ZFR$ proves
that $j:V\to V$ is fully elementary (as a theorem scheme).

Work in $\ZF_2$. 
Then $\kappa\in\OR$ is \emph{Reinhardt}
iff there is a class $j$ such that $(V,j)\sats\ZFR$ and $\kappa=\crit(j)$.
Following \cite{woodin_koellner_bagaria},
 $\kappa\in\OR$ is \emph{super Reinhardt}
 iff for every $\lambda\in\OR$ there is a class $j$
 such that $(V,j)\sats\ZFR$ and $\crit(j)=\kappa$
 and $j(\kappa)\geq\lambda$.

\section{$\RL\Rightarrow$ no parameter-definable $j:V\to V$}\label{sec:RL}

\begin{fact}[Suzuki]\label{fact:suzuki_no_def_j}
Assume $\ZF$. Then no class $j$ which is definable from parameters is an 
elementary
 $j:V\to V$.
\end{fact}

Of course, the theorem is really a theorem scheme,
giving one statement for each possible formula $\varphi$
being used to define $j$ (from a parameter).
We generalize here the proof of \cite[Theorems 
5.6, 5.7***]{cumulative_periodicity},
in order to show Suzuki's fact above
is actually a consequence of lesser theory $\RL$,
which is the basic first 
order theory
modelled by $V_\lambda$ for limit  $\lambda$ (without choice):

\begin{dfn}\label{dfn:RL} $\RL$ (for \emph{Rank Limit}) is the theory in $\Ll$
 consisting of Empty Set, Extensionality,  Foundation,
 Pairing, Union, Power Set, Separation (for all formulas, from 
parameters),
 together with the statements
 ``For every ordinal $\alpha$, $\alpha+1$ exists, $V_\alpha$
 exists and $\left<V_\beta\right>_{\beta<\alpha}$ exists'',
 and ``For every $x$, there is an ordinal $\alpha$
 such that $x\in V_\alpha$''.
\end{dfn}

Since $\RL$ lacks Collection,
  the two extra statements regarding
the cumulative hierarchy at the end are important.
 Clearly a model $M\sats\RL$ can contain objects $R$
 such that $M\sats$``$R$ is a wellorder'',
 but such that there is no $\alpha\in\OR^M$
 such that $M\sats$``$\alpha$ is the ordertype of $R$''.

\begin{tm}\label{tm:RL_no_def_j}Assume $\RL$. Then there is no 
$\Sigma_1$-elementary
 $j:V\to V$ which is definable from parameters.
\end{tm}
\begin{proof}[Proof Sketch]
The proof is a refinement of those of 
\cite[Theorems 5.6, 5.7***]{cumulative_periodicity},
with which we assume the reader is familiar.
By Suzuki's fact, we may assume that $\ZF$ fails.
So either (i) $\OR$ is $\bfSigma_m$-singular for some  $m\in\NN$;
that is, there is an ordinal $\gamma$ and $\Sigma_m$ formula $\varphi$
and parameter $p$ such that $\varphi(p,\xi,\beta)$
defines a cofinal map $\gamma\to\OR$, via $\xi\mapsto\beta$,
and we take then $(\gamma,k)$ then lexicographically
least, so $\gamma$ is regular; 
or (ii) otherwise, and there is $\eta\in\OR$ and  $m\in\NN$
and a $\Sigma_m$ formula $\varphi$ and $p\in V$ such that
$\varphi(p,x,\beta)$ defines a cofinal map $f:V_\eta\to\OR$,
via $x\mapsto\beta$, and we take $\eta$ least such,
which as in  \cite[Remark 5.3***]{cumulative_periodicity} implies 
that $\eta=\alpha+1$ and
that $\rg(f)$ has ordertype $\OR$, and then we may in fact
take $f$ to be surjective, and obtain prewellorders
of $V_{\alpha+1}$ having ordertype $\OR$,
and so define $P=\scot(\OR)$ as before. ($\RL$ suffices
here; for example, starting with 
an arbitrary definable cofinal $f:V_{\alpha+1}\to\OR$,
for each $\beta\in\OR$, $\rg(f)\inter\beta$ is a set, by Separation. And the 
fact that a prewellorder
has ordertype $\OR$ is a first-order statement under $\RL$.)
Let $x_0$ be this $\gamma$ or $P$ respectively.
Note that $x_0$ is definable without parameters
(but maybe not $\Sigma_1$-definable without parameters).

Now as before, given a $\Sigma_1$-elementary $j:V\to V$
which is definable from parameters, we can define 
finite iterate $j_n$,
and
\[ j_n:(V,j_n)\to(V,j_{n+1}) \]
is $\in$-cofinal $\Sigma_1$-elementary. Note  that $\NN$
earlier denotes the \emph{standard} integers,
so $m$ is standard; also all discussion of elementarity
is with standard formulas. However, our $\RL$-model
$V$ might have non-standard integers (from the external perspective),
and when we write ``$\om$''  in what follows,
it is the set of $V$-integers (and $n$ indexing
$j_n$ above is an arbitrary element of $\om$).
Now the sequence 
$\left<j_n\right>_{n<\om}$
is in fact a definable class (show by (internal to $V$) induction on $n<\om$
that for each $\alpha\in\OR^V$, there is
$\beta\in\OR^V$ such that $j_n\rest V_\alpha$ is computable from $j\rest 
V_\beta$).
So basically as in \cite{cumulative_periodicity},
for each $\alpha\in\OR^V$,
there is $n<\om$ such that $j_n(\alpha)=\alpha$ (and hence
$j_m(\alpha)=\alpha$ for $m\geq n$);
however, we must use here the definability of $\left<j_n\right>_{n<\om}$ from 
parameters,
and Separation,
to get that the relevant sequence of sets $\left<A_n\right>_{n<\om}$
is a set; hence we can consider $A=\bigcap_{n<\om}A_n$
and $j(A)$. Similarly, if $\OR$ is regular to class functions
definable from parameters and $P$ is as above, then 
$j_n(P)=P$ for some $n$.

Now given any  $\Sigma_1$-elementary $j:V\to V$,
we have that $j$ is fully elementary iff $j(x_0)=x_0$.
For if $j$ is elementary then $j(x_0)=x_0$ because $x_0$ is outright definable.
The other direction
is proved like in \cite{cumulative_periodicity}, but we must restrict to 
classes $A\sub V$
which are definable from parameters over $V$ (which clearly
suffices for our purposes here), since we need to use
Separation to get that the relevant $x_0$-indexed
sequence of sets $\left<A_y\right>_y$ is a set.

Now assume there is some 
$j:V\to V$ which is $\Sigma_1$-elementary and definable from parameters,
and fix a formula $\varphi$ and parameter $p$ which defines $j$.
So there is such a $(\varphi,p,j)$ satisfying
$j(x_0)=x_0$ (of course, if we need to pass to $j_n$ with $n$ non-standard,
we can incorporate $n$ into $p$).
For $q\in V$ let
\[ j_q=\{(x,y)\bigm|\varphi(q,x,y)\}.\]
Let $\kappa_0\in\OR$ be least such that for some $q$,
$j_q:V\to V$ is $\Sigma_1$-elementary and $\crit(j_q)=\kappa_0$
and $j_q(x_0)=x_0$. Since $x_0$ is outright definable,
so is $\kappa_0$.

Fix $p_0$ witnessing this.
Since $j_{p_0}$ is $\Sigma_1$-elementary and $j_{p_0}(x_0)=x_0$,
$j_{p_0}$ is fully elementary.
But $j_{p_0}(\kappa_0)>\kappa_0$, so $\kappa_0\notin\rg(j_{p_0})$,
a contradiction.
\end{proof}

\section{Low-level 
definability over $L_1(V_\delta)$}\label{sec:j:V_delta_to_V_delta_in_L(V_delta)}

We saw in \S\ref{sec:RL} a generalization of the fact,
shown in \cite[Theorem 5.7***]{cumulative_periodicity}, that 
if $j\in\mathscr{E}_1(V_\delta)$
where $\delta\in\Lim$, then $j$ is not definable from parameters
over $V_\delta$. The sets $X\sub V_\delta$ which are definable
from parameters over $V_\delta$ are exactly those which are in $L_1(V_\delta)$,
or equivalently, in terms of Jensen's hierarchy, those which are in 
$\J(V_\delta)$ (the rudimentary closure of $V_\delta\cup\{V_\delta\}$).
Beyond this, it is natural to consider whether one might get such a 
$j\in\mathscr{E}_1(V_\delta)$ which is constructible from $V_\delta$, i.e.,
in $L(V_\delta)$. This also naturally ramifies: given an ordinal $\alpha$,
can there be such a $j\in\J_\alpha(V_\delta)$ (or $\in L_\alpha(V_\delta)$)?
The question can also of course be extended beyond $L(V_\delta)$.

Later in the paper, we will find some quite precise answers
to some of these questions. But in this section, for a warm-up and for some 
motivation, we consider the simplest instance not covered by
\cite{suzuki_no_def_j} or \cite{cumulative_periodicity}; that is, the question 
of whether
there can be a $j\in\mathscr{E}_1(V_\delta)$
which is at the simplest level of definability beyond definability
from parameters over $V_\delta$.

\begin{rem}\label{rem:J(X)}
We use Jensen's refinement of the $\J$-hierarchy
into the $\Ss$-hierarchy. Here is a summary of the features we need;
the reader can refer to \cite{jensen_fs} or \cite{schindler2010fine} for more 
details.
Recall that for a transitive set $X$,
 \[ \J(X)=\{f(V_\lambda,\vec{x})\bigm|f\text{ is a rudimentary function and 
}\vec{x}\in X^{<\om}\} =\bigcup_{n<\om}\Ss_n(X),\]
where $\Ss$ is Jensen's $\Ss$-operator,
and $\Ss_0(X)=X$ and $\Ss_{n+1}(X)=\Ss(\Ss_n(X))$.
Taking $\Ss$ defined appropriately (maybe not exactly how Jensen originally 
defined it), each $\Ss_n(X)$
is transitive, $\Ss_n(X)\in\Ss_{n+1}(X)$ and so $\Ss_n(X)\sub\Ss_{n+1}(X)$.
So
$\Ss_n(X)\elem_0\Ss_{n+1}(X)\elem_0\J(X)$
and for $\Sigma_1$ formulas $\varphi$ and  $y\in\J(X)$,
\[ \J(X)\sats\varphi(y)\iff\exists n<\om\ 
[y\in\Ss_n(X)\sats\varphi(y)].\]
The truth of $\Sigma_1$ statements over $\J(X)$ also reduces
uniformly recursively to countable disjunctions of statements over $X$.
That is, there is a recursive function 
$(\varphi,\vec{f},n)\mapsto\psi=\psi_{\varphi,\vec{f},n}$
such that for all
$\Sigma_1$ formulas $\varphi$,
tuples $\vec{f}$ of (terms for) rudimentary functions
and $n<\om$, then $\psi=\psi_{\varphi,\vec{f},n}$ is a formula
in $\Ll$, and for all transitive $X$  and $\vec{x}
\in X^{<\om}$, 
\[ \J(X)\sats\varphi(\vec{f}(X,\vec{x}))\iff\exists n<\om\  
[X\sats\psi_{\varphi,\vec{f},n}(\vec{x})].\]
Also, for each rudimentary $f$, the graph
\[\{(\vec{x},y)\bigm|\vec{x}\in X^{<\om}\text{ and }y=f(X,\vec{x})\} \]
is $\Sigma_1$-definable over $\J(X)$ from the parameter $X$,
uniformly in $X$.
For $A\sub X$, we have
$A\in\J(X)$ iff $A\text{ is definable from parameters over }X$.
\end{rem}

\begin{lem}\label{lem:j_extends_to_j^+} \tu{(}$\ZF$\tu{)}
Let $\lambda\in\Lim$ and
 $j\in\mathscr{E}(V_\lambda)$. There is a unique
 $j'\in\mathscr{E}_1(\J(V_\lambda))$
 with $j\sub j'$.
\end{lem}
\begin{proof}
We must set $j'(\lambda)=\lambda$ and
$j'(V_\lambda)=V_\lambda$.
Because
\[ \J(V_\lambda)=\bigcup_{n<\om}\Ss_n(V_\lambda) 
 =\{f(V_\lambda,\vec{x})\bigm|\vec{x}\in V_\lambda^{<\om}\},\]
and since for  rudimentary $f$, the graph
$\{(\vec{x},y)\bigm|\vec{x}\in V_\lambda^{<\om}\text{ and 
}y=f(V_\lambda,\vec{x})\}$
is $\Sigma_1$ over $\J(V_\lambda)$ in the parameter $V_\lambda$,
we must set
$j'(f(V_\lambda,\vec{x})) =f(V_\lambda,j(\vec{x}))$,
 giving uniqueness.

But this definition gives a well-defined 
$j'\in\mathscr{E}_0(\J(V_\lambda))$.
This is by \cite{jensen_fs}: 
all rudimentary functions are simple, and hence
for each $\Sigma_0$ formula  $\varphi$ and all rud functions 
$f_0,\ldots,f_{k-1}$, there is a formula $\varphi'$
such that for all $\vec{x}\in V_\lambda^{<\om}$
and $y_i=f_i(V_\lambda,\vec{x})$, we have
\[ 
\J(V_\lambda)\sats\varphi(\vec{y})\iff
V_\lambda\sats\varphi'(\vec{x}).\]
But by the elementarity of $j$, 
\[V_\lambda\sats\varphi'(\vec{x})\iff V_\lambda\sats\varphi'(j(\vec{x}))\iff
 \J(V_\lambda)\sats\varphi(j'(y_0),\ldots,j'(y_{k-1})).\]
Note $j'$ is also $\in$-cofinal, and hence $\Sigma_1$-elementary.
And $j\sub j'$.
\end{proof}

Note that in the following, 
 $j\rest V_\alpha\in V_\lambda$
for each $\alpha<\lambda$. The following theorem
strengthens \cite[Theorem 5.7 part 1***]{cumulative_periodicity}
(in a different manner than \ref{tm:RL_no_def_j}).
We will actually prove more later (Theorems 
\ref{tm:V=L(V_delta)_cof(delta)>om} and
\ref{tm:no_j_in_adm}), but we start here for some motivation:

\begin{tm}\label{tm:j_not_amenably_Sigma_1}
\tu{(}$\ZF$\tu{)} Let $\lambda\in\Lim$ and
$j\in\mathscr{E}_1(V_\lambda)$. Then
$\widetilde{j}=\{j\rest V_\alpha\bigm|\alpha<\lambda\}$ is not 
$\bfSigma_1^{\J(V_\lambda)}$.
\end{tm}
\begin{proof}
 Suppose otherwise. Note that each finite iterate $j_n$ is then
 likewise definable, so by \cite[Theorem 5.6***]{cumulative_periodicity}, we 
may assume that $j\in\mathscr{E}(V_\lambda)$.
 Let $\kappa$ be the least critical point among all such fully elementary
$j$, select  $j$ witnessing the choice of $\kappa$, and
choose $p\in V_\lambda$ and a $\Sigma_1$ formula $\varphi$
such that 
\[ \widetilde{j}=\{k\in 
V_\lambda\bigm|\J(V_\lambda)\sats\varphi(k,p,V_\lambda)\}.\]
For $n<\om$, let 
$\eta_n=\bigcup\{\alpha<\lambda\bigm|\Ss_n(V_\lambda)\sats\varphi(j\rest 
V_\alpha,p,V_\lambda)\}$.  
By Remark \ref{rem:J(X)} and since $j$ is not definable from parameters over 
$V_\lambda$ (by Theorem \ref{tm:RL_no_def_j} or \cite[Theorem 
5.7***]{cumulative_periodicity}),
it follows that $\eta_n<\lambda$.
Note $\eta_n\leq\eta_{n+1}$ and $\sup_{n<\om}\eta_n=\lambda$.

Let $(q,\mu)\in V_\lambda\cross\lambda$
and $m<\om$.
Say that $(q,\mu)$ is \emph{$m$-good} iff
for each $n\leq m$,
there are $(\xi_n,\xi'_n)\in\lambda^2$ and $\ell_n\in V_\lambda$ such that:
\begin{enumerate}
 \item $\ell_n:V_{\xi_n}\to V_{\xi_n'}$ is $\Sigma_0$-elementary and cofinal
 (where if $\xi_n$ is a successor $\gamma+1$, cofinality
 means that $\xi_n'$ is also a successor and $\ell_n(V_{\xi_n-1})=V_{\xi_n'-1}$)
\item  for all $\Sigma_m$ formulas $\psi$
 and $x\in V_{\xi_n}$ $[V_\lambda\sats\psi(x)$ iff 
$V_\lambda\sats\psi(\ell_n(x))]$,
\item $\Ss_n(V_\lambda)\sats\text{``}\ell_n\text{ is the union of all
}k\text{ such that }\varphi(k,q,V_\lambda)\text{''}$,
 \item $\ell_i\sub\ell_n$ for each $i\leq n$,
 \item $\crit(\ell_n)=\mu$.
\end{enumerate}
If $(q,\mu)$ is $n$-good, write $\ell^{q}_n=\ell_n$ and $\xi^q_n=\xi_n$.

Let also $n<\om$. Say that $(q,\mu)$ is \emph{$(m,n)$-strong}
iff $(q,\mu)$ is $m$-good and $\eta_n\leq\xi^q_m$.
Say that $(q,\mu)$ is \emph{$n$-strong}
iff $\exists m<\om\ [(q,\mu)\text{ is }(m,n)\text{-strong}]$.

Recall $j\in\mathscr{E}(V_\lambda)$ with 
$\crit(j)=\kappa$, 
etc.
Let $\alpha_0<\lambda$ be such that $p\in j(V_{\alpha_0})$.
Let
\[ A_m=\{(q,\mu)\in V_{\alpha_0}\cross\kappa\bigm|(q,\mu)\text{ is 
}m\text{-good}\}.\]

Let $A_\om=\bigcap_{m<\om}A_m$. Note that  
$\left<A_m\right>_{m<\om}\in V_\lambda$.

By Lemma \ref{lem:j_extends_to_j^+}, because $j\in\mathscr{E}(V_\lambda)$, it 
extends uniquely to a 
$\widehat{j}\in\mathscr{E}_1(\J(V_\lambda))$
with $\widehat{j}(V_\lambda)=V_\lambda$. This gives that
$\widehat{j}(\Ss_n(V_\lambda))=\Ss_n(V_\lambda)$ for each $n<\om$,
and we have a fully elementary map
\[ j^*_n= 
\widehat{j}\rest\Ss_n(V_\lambda):\Ss_n(V_\lambda)\to\Ss_n(V_\lambda). \]

\begin{clmfour}
 $A_\om\neq\emptyset$, and moreover, $(p,\kappa)\in j(A_\om)$.
\end{clmfour}
\begin{proof}
Because $j^*_n$ is fully elementary and $A_n\in V_\lambda$, 
note 
\[j(A_n)=j^*_n(A_n)=\{(q,\mu)\in V_{j(\alpha_0)}\cross 
j(\kappa)\bigm|(q,\mu)\text{ is }n\text{-good}\}.\]
But $(p,\kappa)$ is $n$-good, so
$(p,\kappa)\in j(A_n)$.
Also $\left<A_n\right>_{n<\om}\in V_\lambda$,
and
\[ j(A_\om)=j(\bigcap_{n<\om}A_n)=\bigcap_{n<\om}j(A_n), \]
so $(p,\kappa)\in j(A_\om)\neq\emptyset$,
so $A_\om\neq\emptyset$.
\end{proof}

Let
$B_{(m,n)}=\{(q,\mu)\in V_{\alpha_0}\cross\kappa\bigm|(q,\mu)\text{ is 
}(m,n)\text{-strong}\}$,
and
\[ B_n=\{(q,\mu)\in V_{\alpha_0}\cross\kappa\bigm|(q,\mu)\text{ is 
}n\text{-strong}\} \]
and $B_\om=\bigcap_{n<\om}B_n$. These sets are in $V_\lambda$.

\begin{clmfour}
 $A_\om\inter B_\om\neq\emptyset$, and moreover, $(p,\kappa)\in 
j(A_\om\inter B_\om)$.
\end{clmfour}
\begin{proof}
By the previous claim and like in its proof, it suffices to see 
that $(p,\kappa)\in j(B_n)$
for each $n<\om$. But $B_n=\bigcup_{m<\om}B_{(m,n)}$,
so
\[ j(B_n)=\bigcup_{m<\om}j(B_{(m,n)}),\]
and
$j(B_{(m,n)})=j^*_{m}(B_{(m,n)})=$
\[ =\{(q,\mu)\in 
V_{j(\alpha_0)}\cross j(\kappa)\bigm|(q,\mu)\text{ is 
}m\text{-good, with }j(\eta_n)\leq\xi^q_m\}.\]
But there is $m<\om$ with $j(\eta_n)\leq\eta_m=\xi^{p}_m$.
Then
$(p,\kappa)\in j(B_{(m,n)})$,
so $(p,\kappa)\in j(B_n)$, as required.
\end{proof}

Now by the claims, we may pick $(q,\mu)\in A_\om\inter B_\om$.
Let
$\ell=\bigcup_{n<\om}\ell^q_n$.

\begin{clmfour}
 $\ell:V_\lambda\to V_\lambda$ is fully elementary,
 $\crit(\ell)=\mu<\kappa$,
and $\widetilde{\ell}$ is $\bfSigma_1^{\J(V_\lambda)}$.
\end{clmfour}
\begin{proof}
Because $(q,\mu)\in A_\om$, for each $n<\om$,
$\ell^q_n:V_{\xi^q_n}\to V_{\xi_n'}$]
is cofinal $\Sigma_0$-elementary with 
$\crit(\ell^q_n)=\mu$, and $\ell^q_n\sub\ell^q_{n+1}$.
So $\ell$ is a function with domain
$V_\lambda=\bigcup_{n<\om}V_{\xi^q_n}$;
the equality is because $(q,\mu)$ is $n$-strong for each $n<\om$.
So $\ell:V_\lambda\to V_\lambda$.
But because  $(q,\mu)$ is $m$-good for each $m<\om$,
$\ell$ is  fully elementary.
And $\crit(\ell)=\mu$.
Finally note that $\widetilde{\ell}$ is appropriately definable
from the parameter $(q_0,V_\lambda)$.
\end{proof}

But $\mu<\kappa$, so the claim contradicts
 the minimality of $\kappa$.
\end{proof}

One can directly generalize the foregoing argument,
showing that an elementary 
$j:V_\delta\to V_\delta$ cannot appear
in $\J_\alpha(V_\delta)$, for some distance. But especially 
once we get to $\alpha\geq\kappa=\crit(j)$ (or worse,
$\alpha\geq\kappa_\om(j)$),
things are clearly more subtle, because in order to extend
$j$ to $\hat{j}:\J_\alpha(V_\delta)\to\J_\alpha(V_\delta)$,
$\hat{j}$ must move ordinals $\geq\delta$. But a natural
and general way to extend $j$ is through taking an ultrapower
by the extender of $j$. So we treat this topic in the next section.

\section{Ultrapowers and extenders under $\ZF$}\label{sec:ultrapowers_in_ZF}

In order to help us analyse the model $L(V_\delta)$
further, and for more general purposes, we want to be able to deal with 
ultrapowers  via extenders.
This will, for example, assist us in extending embeddings of 
the form $j:V_\delta\to 
V_\delta$ to larger models, such as $L(V_\delta)$.

Without
$\AC$, there are some small technical difficulties here
in the definitions, which we will work through first.
The more significant issues are  those of wellfoundedness
of the ultrapower and (generalized) {\L}o\'{s}' theorem; when we apply 
extenders,
we will usually want to know that these properties hold of the ultrapowers we 
form.

\newcommand{\mr}{\mathrm{mr}}

\begin{dfn} \tu{(}$\ZF$\tu{)}
Write $\approx_{\rank}$ for the equivalence relation determined by rank;
that is, $x\approx_{\rank}y$ iff $\rank(x)=\rank(y)$.
A set $r$ is 
\emph{rank-extensional} iff for all $x,y\in r$ with $x\neq y$ but 
$x\approx_{\rank}y$, there is $z\in r$ with $z\in x\Leftrightarrow z\notin y$.

Let $r$ be rank-extensional. The \emph{membership-rank diagram}
of $r$ is the structure $\mr(r)=(r,{\in}\rest r,\approx_{\rank}\rest r)$.
\end{dfn}

An easy induction on rank gives:
\begin{lem} \tu{(}$\ZF$\tu{)}
Let $r$ be rank-extensional. Then:
\begin{enumerate}[label=--]
 \item 
For all $\alpha\in\OR$,
 $r\inter V_\alpha$ is also rank-extensional.
 \item There is no non-trivial automorphism of $\mr(r)$.
\end{enumerate}
\end{lem}

\begin{dfn} \tu{(}$\ZF$\tu{)}
We say that $r$ is an \emph{index} iff $r$ is finite and rank-extensional.

Note that
we have an equivalence relation on the class of all indices
given by $a\approx b$ iff $\mr(a)\iso\mr(b)$.
The \emph{signature} $\sig(a)$ of an index $a$ is the equivalence class of $a$
with respect to $\approx$.
Note that every signature is represented by some element of $V_\om$.
By selecting in some natural way the minimal such representative,
we may consider $\sig(a)$ as being this element of $V_\om$.
A \emph{signature} is a set $\sig(a)$ for some index $a$.
Given a transitive set $X$, we write $\left<X\right>^{<\om}$ for the set of 
indices $r\sub X$, and given an index $b$, 
$\left<X\right>^b$ denotes $\{r\in\left<X\right>^{<\om}\bigm|
\sig(r)=\sig(b)\}$.

If $a\approx b$ then $\pi^{ab}:\mr(a)\to\mr(b)$ denotes the unique
isomorphism.
Let $a,\widetilde{a},\widetilde{b}$ be indices with $a\sub\widetilde{a}$ and 
$\sig(\widetilde{a})=\sig(\widetilde{b})$.
Then $\widetilde{b}^{\widetilde{a}a}$ denotes
$\pi^{\widetilde{a}\widetilde{b}}``a$.
\end{dfn}

\begin{lem}\label{lem:indices_cofinal} \tu{(}$\ZF$\tu{)} For every finite 
set $c$
there is an index $b$
with $c\sub b$ and $\rank(c)=\rank(b)$.\footnote{\label{ftn:rank-ext}
 One might have expected (as did the author initially) that
 for every finite extensional set $a$, there is a finite extensional set $b$
 with $a\sub b$ (where \emph{extensional} means that for all $x,y\in a$
 with $x\neq y$, there is $z\in a$ with $z\in x$ iff $z\notin y$). In fact,
in the first draft of the notes \cite{reinhardt_non-definability} (v1 on 
arxiv.org), which contained the first version of the development of extenders 
here, we  defined \emph{index}
with extensionality replacing rank-extensionality,
and we made precisely that claim.
But that claim is false; here is a counterexample. 
Define sets $n'$ as follows:
 \begin{enumerate}[label=--] 
  \item $0'=\emptyset$,
  \item $1'=\{0'\}$ and $2'=\{1'\}$,
  \item $3'=\{0',2'\}$ and $4'=\{1',3'\}$,
  \item $(2n+1)'=\{0',2',\ldots,(2n)'\}$ and $(2n+2)'=\{1',3',\dots,(2n+1)'\}$.
 \end{enumerate}
Let $x=\{(2n)'\bigm|n\in\om\}$ and $y=\{(2n+1)'\bigm|n\in\om\}$.
 Then $p=\{x,y\}$ is finite but there is no finite extensional $q$ with $p\sub 
q$. (Let $p\sub q$ with $q$ finite and consider the largest $k\in\om$ such that 
$k'\in q$.
Observe that either $x\inter q=k'\inter q$ or $y\inter q=k'\inter 
q$, and hence $q$ is not extensional.)}
\end{lem}
\begin{proof}
 This is an induction on rank (recall though that we don't assume $\AC$).
 Assume $c\neq\emptyset$. Let $\alpha$ be the maximum rank of elements 
of $c$, so $\rank(c)=\alpha+1$. Let $c_{\mathrm{max}}$ be the set of elements 
of $c$ of rank $\alpha$.

 First choose a finite set $c'$ such that $c\sub c'$, all elements of
 $c'\cut c$ have rank ${<\alpha}$, and with $c'$ extensional with respect 
to $c_{\mathrm{max}}$ (that is, for all $x,y\in c_{\mathrm{max}}$ with $x\neq 
y$, there is $z\in c'$
with $z\in x\Delta y$). Note that $\rank(c'\cut c_{\mathrm{max}})\leq\alpha$.
So by 
induction we can fix an index
$b'$ with $c'\cut c_{\mathrm{max}}\sub b'$ and $\rank(b')\leq\alpha$. Now note
that $b=b'\cup 
c_{\mathrm{max}}$ is as desired.
\end{proof}

\newcommand{\ambl}{\mathrm{ambl}}

\begin{dfn}[Extender, Ultrapower]\label{dfn:extender}
\tu{(}$\ZF$\tu{)} Let $M\sats\RL$ be transitive. We say that
$A$ is \emph{amenable to $M$}, and write $A\sub_{\ambl}M$,
iff $A\sub M$ and $A\inter x\in M$ for each $x\in M$.
We also write $A\sub_{\ambl}\left<M\right>^{<\om}$
to mean $A\sub_{\ambl}M$ and $A\sub\left<M\right>^{<\om}$.

Let $M,N$ be transitive, satisfying $\RL$,
and $j:M\to N$ be $\Sigma_1$-elementary and $\in$-cofinal.
The \emph{extender derived from 
$j$}
is the set $E$ of all pairs $(A,a)$
such that
$a\in\langle N\rangle^{<\om}$, $A\sub_{\ambl}M$
and $a\in j(A)$.
Given $a\in\langle N\rangle^{<\om}$,
let $E_a=\{A\bigm|(A,a)\in E\}$.\footnote{Clearly
$(A,a)\in E$ iff $(A\inter V_\xi^M,a)\in E$
where $\xi$ is any ordinal in $M$ such that $a\in j(V_\xi^M)$.
And given the manner
in which $E$ will be used, we could have actually
added the extra demand that $A\in M$ to the requirements
specifying when $(A,a)\in E$, and in terms of information
content and cardinality, it would be more natural to do so.
But it is convenient in other ways to allow more arbitrary amenable sets $A$.}

Let $A\sub_{\ambl}\left<M\right>^{<\om}$ and $a,\widetilde{a}$ be indices
with $a\sub\widetilde{a}$. Then $A^{a\widetilde{a}}$
denotes the set of all $u\in\left<M\right>^{\widetilde{a}}$
such that $u^{\widetilde{a}a}\in A$, and $A^{\widetilde{a}a}$
denotes the set of all $u\in\left<M\right>^{a}$
such that there is $v\in A\inter\left<M\right>^{\widetilde{a}}$
and $u=v^{\widetilde{a}a}$.

Let $f:\langle M\rangle^{<\om}\to V$. Let $a,\widetilde{a}\in\langle 
N\rangle^{<\om}$
with $a\sub\widetilde{a}$.  We define a function 
$f^{a\widetilde{a}}:\left<M\right>^{<\om}\to V$
as follows.
Let 
$u\in\langle M\rangle^{<\om}$.
If $\sig(u)\neq\sig(\widetilde{a})$ then $f^{a\widetilde{a}}(u)=\emptyset$,
and if $\sig(u)=\sig(\widetilde{a})$ then 
$f^{a\widetilde{a}}(u)=f(u^{\widetilde{a}a})$.

Given a transitive rudimentarily closed structure $P$ in the language of set 
theory
(or possibly in a larger language), we say that $E$ is an 
\emph{extender
over $P$} iff $\OR^M<\OR^P$ and $V_{\OR^M}^P=M$.
Suppose $E$ is over $P$.
A \emph{$P$-relevant pair} (with respect to $E$) is a pair 
$(a,f)$ such that
$a\in\langle N\rangle^{<\om}$ and $f\in P$ and
$f:\left<M\right>^{<\om}\to P$.
We define the (internal) \emph{ultrapower $\Ult_0(P,E)$ of $P$ by $E$}.
We first define an equivalence relation $\approx_E$ on the class of 
$P$-relevant pairs, by setting
$(a,f)\approx_E(b,g)$
iff for some/all  $c\in\langle N\rangle^{<\om}$ with $a\cup 
b\sub c$,
we have
\[ \{u\in\langle M\rangle^{<\om}\bigm|f^{ac}(u)=g^{bc}(u)\}\in E_c.\]
 We write $[a,f]^{P,0}_E$ for the equivalence class of $(a,f)$.
We define the relation
$\in_E$
likewise, replacing the condition ``$f^{ac}(u)=g^{bc}(u)$''
with ``$f^{ac}(u)\in f^{bc}(u)$''. Let $\univ{U}$ be the class of
equivalence classes of $P$-relevant pairs
with respect to $\approx_E$, and $\in'$ be the relation on $\univ{U}$
induced by $\in_E$. Then  the \emph{\tu{(}internal\tu{)} 
ultrapower}.\footnote{The ``sub-$0$'' in ``$\Ult_0$''
and the ``super-$0$'' in ``$i^{P,0}_E$'' indicates the internality 
of the 
ultrapower,
i.e. that the functions $f$ used in forming the ultrapower all belong to 
$P$. This is an artifact of related notation in inner model theory,
where one can have $\Ult_n$ for $n\leq\om$.}
$\Ult_0(P,E)$
of $P$ by $E$ is the structure $U=(\univ{U},\in')$.
If $U$ is extensional and wellfounded then we identify it with its transitive 
collapse.
We define the associated \emph{ultrapower embedding}
$i^{P,0}_E:P\to U$ by
\[i^{P,0}_E(x)=[(\emptyset,c_x)]^{P,0}_E\]
where $c_x:\left<M\right>^{<\om}\to P$ is the constant function
$c_x(u)=x$. We often abbreviate this by $i^P_E$ or $i_E$.
For any set $y$, let $\ind(y)$ be the unique index with universe $\{y\}$,
and let $\elmt(\ind(y))=y$ and $\elmt(u)=\emptyset$ if $u$ is not of form 
$\ind(y)$. We write $\spt(E)=N$.
\end{dfn}

Of course, if $P\not\sats\AC$ then the proof of {\L}o\'{s}' theorem
does not go through in the usual manner, so in general $\Ult_0(P,E)$
might not even be extensional.

\begin{lem} \tu{(}$\ZF$\tu{)}
With notation as in Definition \ref{dfn:extender}, we have:
 \begin{enumerate}
 \item\label{item:proj_meas_1} Let $a,\widetilde{a}\in\left<N\right>^{<\om}$ 
with $a\sub\widetilde{a}$
 and $A\sub_{\ambl}\left<M\right>^{<\om}$. Then:
 \begin{enumerate}
 \item If $A\sub\left<M\right>^{a}$ and $B=\left<M\right>^{a}\cut A$
 then $B^{a\widetilde{a}}=\left<M\right>^{\widetilde{a}}\cut 
A^{a\widetilde{a}}$.
\item $E_a$ is an ultrafilter over the set of all 
$A\sub_{\ambl}\left<M\right>^{<\om}$
and $\left<M\right>^a\in E_a$,
and in fact, $\left<V_\xi^M\right>^a\in E_a$ for each 
$\xi<\OR^M$ with $a\sub j(V_\xi^M)$.
\item $A\in E_a$ iff $A^{a\widetilde{a}}\in E_{\widetilde{a}}$.
\item If $A\in E_{\widetilde{a}}$ then $A^{\widetilde{a}a}\in 
E_a$.\footnote{But note that here the converse does not have to hold.}
\end{enumerate}
\item\label{item:f^ab_comm} $f^{ac}=(f^{ab})^{bc}$ for all 
$a,b,c\in\left<N\right>^{<\om}$
with $a\sub b\sub c$ and all functions $f$.
  \item\label{item:indep_of_xi,c} In the definition
  of $\approx_E$ and $\in_E$,
  the choice of $c$ is irrelevant.
  \item $\approx_E$ is an equivalence relation
  on the $P$-relevant pairs,
  \item\label{item:in_respects_approx} $\in_E$ respects $\approx_E$.
 \item\label{item:N_sub_wfp}  $N\sub\wfp(U)$ and for each $\beta<\OR^N$
we have
$V_\beta^{U}=V_\beta^N$. Moreover, for $x\in V_\beta^N$, we have
$x=[\ind(x),\elmt]^{P,0}_E$.
 \end{enumerate}
\end{lem}
\begin{proof}
Parts \ref{item:proj_meas_1} and \ref{item:f^ab_comm}
are straightforward.
Part \ref{item:indep_of_xi,c}: Consider $\approx_E$,
and pairs $(a,f),(b,g)$.
Take $c,c'\in\left<N\right>^{<\om}$ with $a\cup b\sub c,c'$.
Note we may assume $c\sub c'$. We must see 
$A\in E_c$ iff $A'\in E_{c'}$ where
\begin{enumerate}[label=--]
 \item $A=\{u\in\langle 
V_\alpha^M\rangle^c\bigm|f^{ac}(u)=g^{bc}(u)\}$,
\item $A'=\{u\in\langle 
V_{\alpha}^M\rangle^{c'}\bigm|f^{ac'}(u)=g^{bc'}(u)\}$.
\end{enumerate}
As $c\sub c'$, part \ref{item:f^ab_comm} gives that
$A'=A^{cc'}$, so $A\in E_c$ iff $A'\in E_{c'}$ by
part \ref{item:proj_meas_1}.
The rest of parts 
\ref{item:indep_of_xi,c}--\ref{item:in_respects_approx} is similar or follows 
easily.

Part \ref{item:N_sub_wfp}:
Easily from the definitions, for $x,y\in N$ we get
\[ (\ind(x),\elmt)\approx_E(\ind(y),\elmt)\iff x=y, \]
\[ (\ind(x),\elmt)\in_E(\ind(y),\elmt)\iff x\in y.\]
So let $(a,f)$ be a $P$-relevant pair and $x\in N$ be such that
\[ (a,f)\in_E(\ind(x),\elmt). \]
Note we may assume  $x\in a$ and there is $\xi<\OR^M$
such that $a\in j(V_\xi^M)$ and $\rg(f)\sub V_\xi^M$,
and so $g=f\rest V_\xi^M\in M$.
But then $(a,f)\approx_E(a,g)$ and
\[ j(g)(a)\in j(\elmt\rest V_\xi^M)(\ind(x))=x,\]
so $j(g)(a)=y$ for some $y\in x$. But then 
$(a,f)\approx_E(\ind(y),\elmt)$,
as desired.
\end{proof}

\begin{dfn} \tu{(}$\ZF$\tu{)}
 Let $E$ be an extender over a transitive rudimentarily closed
 structure 
$M$.
 We say that \emph{$\Sigma_0$-{\L}o\'{s}' criterion} holds for $\Ult_0(M,E)$
 iff for all $n<\om$, for all $f_0,f_1,\ldots,f_n\in M$,
 for all $a\in\left<\spt(E)\right>^{<\om}$,
 and all $\Sigma_0$ formulas $\varphi$, if there are
 $E_a$-measure one many $u\in\left<M\right>^{<\om}$
 such that
 \[ M\sats\exists y\in f_0(u) [\varphi(f_1(u),\ldots,f_n(u),y)] \]
 then there is $b\in\left<\spt(E)\right>^{<\om}$
 and $g\in M$ 
 such that $a\sub b$ and for $E_b$-measure one many
 $v$, we have
 \[ M\sats g(v)\in 
f_0^{ab}(v)\text{ and }\varphi(f_1^{ab}(v),\ldots,f_n^{ab}(v),g(v)).\]

We define \emph{{\L}o\'{s}' criterion}
for $\Ult_0(M,E)$ analogously, but we allow arbitrary formulas 
$\varphi$, and the $\exists$ quantifier is unbounded.
\end{dfn}

\begin{tm}[Generalized {\L}o\'{s}' Theorem]\label{tm:generalized_Los}
\tu{(}$\ZF$\tu{)} Let $M$ be a transitive rudimentarily closed structure and 
$E$ 
be an 
extender over $M$.
 Suppose that $\Sigma_0$-{\L}o\'{s}' criterion holds for $\Ult_0(M,E)$.
 Then for all $n<\om$, all $f_1,\ldots,f_n\in M$,
  all $a_1,\ldots,a_n\in\left<\spt(E)\right>^{<\om}$
 and all $\Sigma_0$ formulas $\varphi$, letting 
$a\in\left<\spt(E)\right>^{<\om}$
 be such that $a_i\sub a$ for each $i$, we have
 \[ \Ult_0(M,E)\sats\varphi([a_1,f_1],\ldots,[a_n,f_n]) \]
 iff there are $E_a$-measure one many $v\in\left<M\right>^{<\om}$ such that
 \[ M\sats\varphi(f_1^{a_1b}(v),\ldots,f_n^{a_nb}(v)).\] 
  Therefore, the ultrapower embedding $i^{M,0}_E$
is $\in$-cofinal and $\Sigma_1$-elementary.
Moreover, if  {\L}o\'{s}' criterion holds for $\Ult_0(M,E)$,
then the above equivalence holds for arbitrary formulas $\varphi$,
and $i^{M,0}_E$ is fully elementary.
\end{tm}
\begin{proof}
 This is basically the usual induction to prove {\L}o\'{s}' theorem
 under $\AC$, except that we use {\L}o\'{s}' requirement
instead of appealing to $\AC$. One difference, however,
is that we need to allow 
enlarging $a$
to $b$ in order to find an element $[b,g]$ of the ultrapower
witnessing the a statement; in the usual proof of {\L}o\'{s}' theorem,
one can take $a=b$.
\end{proof}

\section{Definability of $V$-criticality and wLS 
cardinals}\label{sec:crit_def}

Under $\ZFC_2$, the
fact that $\kappa=\crit(j)$ for some elementary $j:V\to M$
with $M$ transitive, is  equivalent to the measurability
of $\kappa$.
Therefore this ``$V$-criticality'' of $\kappa$ is first-order. We 
make a brief digression to consider this question under $\ZF_2$.

\begin{dfn}
 \tu{(}$\ZF_2$\tu{)} An ordinal $\kappa$ is $V$-\emph{critical}
 iff there is an elementary $j:V\to M$ with $\crit(j)=\kappa$,
 where $M\sub V$ is transitive.
\end{dfn}

\begin{lem}\label{lem:critical_inaccessible}
 \tu{(}$\ZF_2$\tu{)} Let $\kappa$ be $V$-critical.
 Then $\kappa$ is inaccessible.
\end{lem}
\begin{proof}
 Suppose not and let $\alpha<\kappa$ and $f:V_\alpha\to\kappa$
 be cofinal. Let $j:V\to M$ be elementary with $\crit(j)=\kappa$.
 Then $j(f)=j\com f=f$, although by elementarity, $j(f):V_\alpha\to j(\kappa)$ 
is cofinal, a contradiction.
\end{proof}

We do not know whether $\ZF_2$ proves that $V$-criticality is a first-order
property. But we will show that  $\ZF_2$+``There is a proper class 
of weak L\"owenheim-Skolem cardinals'' does prove this.
Recall this notion from \cite[Definition 4]{usuba_ls}:

\begin{dfn}[Usuba]\label{dfn:wLS}
($\ZF$) A cardinal $\kappa$ is \emph{weak 
L\"owenheim-Skolem} (\emph{wLS})  if for every
$\gamma<\kappa$ and $\alpha\in[\kappa,\OR)$ 
and $x\in V_\alpha$,
there is $X\elem V_\alpha$
with $V_\gamma\sub X$, $x\in X$ and 
the transitive collapse of $X$ in $V_\kappa$.
\end{dfn}

\begin{rem}\label{rem:LS}
Usuba also defines \emph{L\"owenheim-Skolem} (LS)
cardinals, which is at least superficially stronger.
As Usuba mentions in \cite{usuba_ls}, $\ZFC$ proves
that there is a proper class of LS cardinals,
and that (a result of Woodin is that) under just $\ZF$, every 
supercompact cardinal is an LS cardinal.
Thus, assuming $\ZF$ and that there is a super Reinhardt cardinal,\footnote{In
\cite[v1]{reinhardt_non-definability}, it asserted that $\ZF+$ a proper class
of Reinhardt cardinals proves there is a proper class of LS cardinals,
but this should have been a super Reinhardt (which easily implies a proper 
class of the same).}
then there is a proper class of LS cardinals,
and hence wLS cardinals. The next lemma is immediate, due to Usuba:
\end{rem}

\begin{lem}\label{lem:wLS_cards_closed}
\tu{(}$\ZF$\tu{)} We  have:
\begin{enumerate}
 \item The class of wLS cardinals is closed.
 \item Suppose there is a proper class of wLS cardinals
 and let $\gamma\in\OR$ be regular.
 Then there is a proper class of wLS cardinals $\delta$
 such that $\cof(\delta)=\gamma$.
\end{enumerate}
\end{lem}

\begin{dfn}
\tu{(}$\ZF$\tu{)} A \emph{$V$-criticality pre-witness}
is a tuple  $(\kappa,\delta,N,j)$ 
such that $\delta$ is a weak L\"owenheim-Skolem
 cardinal, $N\sats\RL$ is transitive, $j:V_\delta\to N$ is $\in$-cofinal and 
$\Sigma_1$-elementary and $\crit(j)=\kappa$.
\end{dfn}

\begin{tm}\label{tm:cof(delta)>om_critical}
\tu{(}$\ZF$\tu{)} Let $(\kappa,\delta,N,j)$ be a $V$-criticality pre-witness.
Let $U=\Ult(V,E_j)$.
Then:
\begin{enumerate}\item\label{item:wLS_Los} {\L}o\'{s}' criterion holds 
for $U$, 
so $U$ 
is extensional
and $i_E$ is elementary.
and \item\label{item:cof(delta)>om_wfd} If $\cof(\delta)>\om$ then $U$ is 
wellfounded and $j\sub i_E$
and $\kappa=crit(i_E)$ is $V$-critical.
\end{enumerate}
\end{tm}

\begin{proof}
Part \ref{item:wLS_Los}:
Let  $a\in \left<N\right>^{<\om}$
and $f:\left<V_\delta\right>^{<\om}\to V$ 
and $\varphi$ be a formula
and suppose that for $E_a$-measure one many $u\in\left<V_\delta\right>^{<\om}$,
there is $y$ such that $\varphi(f(u),y)$.

Let $n<\om$ be large and $\alpha\in\OR$ be large and with $V_\alpha\elem_n V$.
Let $\gamma<\delta$ be such that $a\sub j(V_\gamma)$.
So $\left<V_\gamma\right>^{<\om}\in E_a$.
Applying weak L\"owenheim-Skolemness at $\delta$,
let $X\elem V_\alpha$ with
\[ V_\gamma\cup\{\gamma,f,\delta,j,N,E\}\sub X \]
and such that the transitive collapse $\bar{X}$ of $X$
is in $V_\delta$. Let $\pi:\bar{X}\to X$ be the uncollapse map.
Let $\pi(\bar{f})=f$, etc.

By the elementarity, for each $u\in\left<V_\gamma\right>^{<\om}$,
and each $v\in\bar{X}$, we have that $\bar{X}\sats\varphi(\bar{f}(u),v)$
iff $V\sats\varphi(f(u),\pi(v))$.

Note
we can fix $y\in j(\bar{X})$
such that $j(\bar{X})\sats\varphi(j(\bar{f})(a),y)$.
Let $\xi\in(\gamma,\delta)$ with $\bar{X}\in V_\xi$
and $b\in\left<j(V_\xi)\right>^{<\om}$ with $a\cup\{a,y\}\sub b$.
Then for $E_b$-measure one many $w\in\left<V_\xi\right>^{<\om}$,
letting $(a^w,y^w)=\pi^{bw}(a,y)$, we have that 
$a^w\in\left<V_\gamma\right>^{<\om}$
and $y^w\in\left<\bar{X}\right>^{<\om}$
and $\bar{X}\sats\varphi(\bar{f}(a^w),y^w)$,
and hence $V\sats\varphi(f(a^w),\pi(y^w))$.

So define $g:\left<V_\xi\right>^{<\om}\to V$
by setting $g(w)=\pi(y^w)$ for all such $w$ (and $g(w)=\emptyset$ otherwise).
Then for $E_b$-measure one many
$w\in\left<V_\xi\right>^{<\om}$, we have $V\sats\varphi(f^{ab}(w),g(w))$,
as desired.

Part \ref{item:cof(delta)>om_wfd}: Suppose not. So $\cof(\delta)>\om$.
For limit ordinals $\xi<\delta$, let $E_\xi$
be the extender derived from
\[ j\rest V_\xi:V_\xi\to V_{\sup j``\xi}.\]
Let $U_\xi=\Ult_0(V,E_\xi)$ and
$k_\xi:U_\xi\to\Ult_0(V,E)$
the natural factor map and $k_{\xi\zeta}:U_\xi\to U_\zeta$
likewise. We have not verified {\L}o\'{s}' criterion
for these partial ultrapowers, so we do not claim elementarity of the maps;
nor do we claim that $U_\xi$ is extensional.
But note that $k_\xi$ and $k_{\xi\zeta}$ are well-defined
and respect ``$\in$'' 
and ``$=$'' (that is, in the sense of the ultrapowers,
even if they fail extensionality),
and commute; that is, $k_{\xi\beta}=k_{\zeta\beta}\com k_{\xi\zeta}$.

Let $\xi\leq\delta$ be  a limit.
Let $\mathscr{O}_\xi=\OR^{U_\xi}=\Ult_{V}(\OR,E_\xi)$,
with the notation meaning that we use all functions (in $V$) which map into 
$\OR$ to form the ultrapower.
Each $\mathscr{O}_\xi$ is a linear order.
Now $U_\xi$ is wellfounded
iff $\mathscr{O}_\xi$ is wellfounded
(see \S\ref{subsec:notation}).
In particular,
$\mathscr{O}_\delta$ is
illfounded.
By restricting 
$k_{\xi\zeta}$, we get a commuting system of order-preserving maps
$\ell_{\xi\zeta}:\mathscr{O}_\xi\to\mathscr{O}_\zeta$
(so $\ell_{\xi\zeta}\sub k_{\xi\zeta}$
and $\ell_{\xi\beta}=\ell_{\zeta\beta}\com\ell_{\xi\zeta}$).
Note that the direct limit of the $\mathscr{O}_\xi$ under the 
maps $\ell_{\xi\zeta}$,
for $\xi\leq\zeta<\delta$, is 
isomorphic to $\mathscr{O}_\delta$,
$\ell_{\xi\delta}$ is the direct limit map.
Note that each $\ell_{\xi\zeta}$ is cofinal.

We claim there is $\xi<\delta$ such that $\mathscr{O}_\xi$ is illfounded
(here we use $\cof(\delta)>\om$).
For suppose not. Then $\mathscr{O}_\xi\iso\OR$.
Define a sequence $\left<\xi_n,\eta_n\right>_{n<\om}$ of pairs of ordinals.
Let $\xi_0=0$. Now $\ell_{0\delta}$ is cofinal.
Let $\eta_0$ be the least  $\eta$
with
$\ell_{0\delta}(\eta)$ in the illfounded part of $\mathscr{O}_\delta$.
Given $(\xi_n,\eta_n)$
with $i_{\xi_n\delta}(\eta_n)$  in the illfounded part,
 there is a pair $(\xi,\eta)$
such that $i_{\xi_n\xi}(\eta_n)>\eta$
and $i_{\xi\delta}(\eta)$ in the illfounded part.
Let $(\xi_{n+1},\eta_{n+1})$ be lexicographically least such.
Let $\xi=\sup_{n<\om}\xi_n$. Because $\cof(\delta)>\om$,
we have $\xi<\delta$. But the sequence just constructed
exhibits that  
$\mathscr{O}_\xi$
is illfounded, a contradiction.

So fix $\xi<\delta$ with $\mathscr{O}_\xi$  illfounded.
Let $n<\om$ be large, let $\alpha\in\OR$ be large
with $V_\alpha\elem_nV$; hence, for some 
$\beta<\alpha$, we have $V_\alpha\sats$``$\Ult_0(V_{\beta},E_\xi)$
is illfounded''. 
Let $\xi'=\sup j``\xi$.
Using the weak L\"owenheim-Skolemness of $\delta$,
let $X\elem V_\alpha$ with
\[ V_\xi\cup V_{\xi'}^N\cup\{N,j,\xi,E_\xi,\beta\}\sub X\]
and the transitive collapse $\bar{X}$ of $X$ in $V_\delta$.
So letting $\pi:\bar{X}\to X$ be the uncollapse map, 
we have $\pi(E_\xi)=E_\xi$, $\pi\rest V_\xi=\id$, etc.
And $\bar{X}\sats$``$\Ult_0(V_{\bar{\beta}},E_\xi)$
is illfounded'', where $\pi(\bar{\beta})=\beta$.
As $\bar{X}$ is transitive and models enough of $\ZF$, it follows that
$\bar{U}=\Ult_0(V_{\bar{\beta}}^{\bar{X}},E_\xi)$
 is illfounded.

Now define
$\sigma:\bar{U}\to N$
by setting
$\sigma([a,f])=j(f)(a)$.
(This makes sense, as $f\in\bar{X}\in V_\delta=\dom(j)$.)
Then note that (since $E_\xi$ is derived from $j$), $\sigma$
is $\in$-preserving. But then $N$ is illfounded,
contradicting our assumption that $N$ is transitive. So $U$ is wellfounded, 
as desired.
This completes the proof of the theorem.
\end{proof}

Of course under $\ZFC_2$, $V$-criticality is equivalent to measurability,
and has a first-order formulation. We can now generalize this result:

\begin{tm}\label{tm:crit_def} \tu{(}$\ZF_2$\tu{)} Assume  a proper class
 of wLS cardinals. Let $\kappa\in\OR$. Then the following
 are equivalent:
\begin{enumerate}[label=--]
\item $\kappa$ is $V$-critical
\item there is a $V$-criticality pre-witness $(\kappa,\delta,N,j)$
with $\cof(\delta)>\om$,
\item there is a $V$-criticality pre-witness $(\kappa,\delta,N,j)$
such that $\Ult_0(V,E_j)$ is wellfounded, where 
$E_j$ is the extender
derived from $j$.
\end{enumerate}
In particular, $V$-criticality is first-order definable over $V$.
\end{tm}
\begin{proof}
Suppose first that $\kappa$ is $V$-critical, and let $k:V\to M$
be elementary with $\crit(k)=\kappa$.
By Lemma \ref{lem:critical_inaccessible}, $\kappa$ is regular.
So by hypothesis and Lemma \ref{lem:wLS_cards_closed},
we can fix a
L\"owenheim-Skolem
cardinal $\delta>\kappa$ with $\cof(\delta)>\om$ (in fact 
$\cof(\delta)=\kappa$ is possible). Let $\xi=\sup k``\delta$
and $N=V_{\xi}^M$ and $j=k\rest V_\delta$.
Then $(\kappa,\delta,N,j)$ is a $V$-criticality pre-witness
with $\cof(\delta)>\om$. We will show below
that it follows that $\Ult_0(V,E_j)$ is wellfounded,
but here it is easier: define
\[ \ell:\Ult_0(V,E_j)\to M, \]
\[ \ell([a,f]^V_{E_j})=k(f)(a),\]
which, directly from the definition of ``$\in$''
and ``$=$'' in the ultrapower, preserves
membership and equality.
But $M$ is transitive, so $\Ult_0(V,E_j)$ is wellfounded.
 Note that we have not yet proved
that {\L}o\'{s}' theorem holds, or even that $\Ult_0(V,E_j)$ is extensional.

Now suppose that $(\kappa,\delta,N,j)$ is any $V$-criticality pre-witness
such that either $U=\Ult_0(V,E)$ is  wellfounded, where $E=E_j$,
or $\cof(\delta)>\om$. Then by Theorem \ref{tm:cof(delta)>om_critical},
 $U$ is 
wellfounded 
and {\L}o\'{s}' criterion holds for this ultrapower,
and hence the ultrapower map $k:V\to U$
is elementary by Generalized {\L}o\'{s}' Theorem, so $U$ is extensional,
so we take $U$ transitive, and
$j\sub k$, so $\crit(k)=\kappa$.
\end{proof}

\begin{cor} \tu{(}$\ZF$\tu{)}
Let $\kappa\in\OR$. Then $\kappa$ is ``definably $V$-critical''
 (that is, witnessed by some definable-from-parameters class $k:V\to M$)
 iff there is a $V$-criticality pre-witness $(\kappa,\delta,N,j)$
 such that $\cof(\delta)>\om$.
\end{cor}
\begin{proof}
 Just use the proof of Theorem \ref{tm:crit_def},
 but now all relevant classes are definable from parameters.
\end{proof}

\begin{cor}\tu{(}$\ZF_2$\tu{)} If $\kappa$ is super Reinhardt,
 then there is a normal measure on $\kappa$ concentrating
 on $V$-critical ordinals.
\end{cor}
\begin{proof}
 If there is a super Reinhardt cardinal then
 there is a proper class of weak L\"owenheim Skolem cardinals,
 by
 Remark \ref{rem:LS}, so the theorem applies, and easily yields the corollary.
\end{proof}

\begin{ques}
\tu{(}$\ZF_2$\tu{)} Suppose $\kappa$ is Reinhardt.
 Is  $V$-criticality first-order?
 Must there be a $V$-critical ordinal ${<\kappa}$?
\end{ques}

Of course if $V$-criticality is first-order and $\kappa$ is Reinhardt, then 
like 
before,
there are unboundedly many $V$-critical ordinals ${<\kappa}$.

\begin{rem}\label{rem:no_R_in_L(X)}One can now easily observe
that if $(V,j)\sats\ZFR$ then there is no set $X$ such that 
$V=L(X)$. Actually much more than this is known,\footnote{Assume 
$\ZFR$.
Goldberg showed a few years ago (unpublished at the time)  that every 
set has a sharp.
The author showed in \cite{reinhardt_non-definability} that
$V\neq\HOD(X)$ for every set $X$. Goldberg \cite{goldberg_email} and Usuba
\cite{usuba_email}
then both (independently) sent proofs to the author that there is no 
set-forcing extension satisfying $\AC$. The author then showed that
$M_n^\#(X)$ exists
for every set $X$, but the proof is not yet published. Much more
local proofs of sharp existence
can now be seen in \cite[***\S6.2]{goldberg_even_ordinals}
and \cite[***\S6]{con_lambda_plus_2}.}
but here is the proof of this simpler fact:
Suppose otherwise. Then there is a proper class of wLS cardinals.
Let $\delta\in\OR$
 be such that $\cof(\delta)>\om$ and $X\in V_\delta$
 and $j``\delta\sub\delta$. Let $E$ be the extender derived from $j\rest 
V_\delta:V_\delta\to V_\delta$.
Then $(\crit(j),\delta,V_\delta,j\rest V_\delta)$ is a $V$-criticality 
pre-witness. 
So by Theorem \ref{tm:crit_def},
$U=\Ult_0(V,E)$ is extensional and wellfounded,
the ultrapower map $k:V\to U$ is elementary,
and $V_\delta\sub U$.
So in fact $U=L(V_\delta)=V$.
But $k$ is definable from the parameter $E$, contradicting Suzuki 
\ref{fact:suzuki_no_def_j}.\footnote{There is actually
a more efficient argument available here,
which we had used in version v1 of 
\cite{reinhardt_non-definability}: Instead of arguing via
Theorem \ref{tm:cof(delta)>om_critical}, one can directly establish
{\L}o\'{s}' theorem and wellfoundedness for the ultrapower,
using that it  is a factor of $j$; for this, we don't need to take 
$\cof(\delta)>\om$.}
\end{rem}

\section{$L(V_\delta)$ and uncountable cofinality}

\begin{lem}\label{lem:inacc_HOD(Vd)_wLS}
\tu{(}$\ZF$\tu{)} Assume  $\delta$ is  
inaccessible and $V=\HOD(V_\delta)$. Then $\delta$ is wLS.
In fact, for all $\alpha\in(\delta,\OR)$
and $p\in V_\alpha$ and $\beta<\delta$
there is $(X,\bar{\delta},\pi)$ such that $\delta,p\in X\elem V_\alpha$ and 
 $\beta\leq\bar{\delta}<\delta$
and $X\inter V_\delta=V_{\bar{\delta}}$
and $\pi:V_{\bar{\delta}}\to X$ is a surjection.
\end{lem}
\begin{proof}
 Given $n<\om$, we may assume that $V_\alpha\elem_n V$, by increasing $\alpha$
 and then subsuming the old $\alpha$ into the parameter $p$.
So since $V=\HOD(V_\delta)$, we may assume
\begin{equation}\label{eqn:V_alpha_is_hull} 
V_\alpha=\Hull_{\Sigma_2}^{V_\alpha}(\OR\cup V_\delta).\end{equation}
Let $X'=\Hull^{V_\alpha}(V_\beta\cup\{p,\delta\})$
and $\bar{\delta}=\sup(X'\inter\delta)$
 and $X=\Hull^{V_\alpha}(V_{\bar{\delta}}\cup\{p,\delta\})$.

 By the inaccessibility of $\delta$, we have $\bar{\delta}<\delta$.
So it suffices to see that $X\elem V_\alpha$ and $X\inter 
V_\delta=V_{\bar{\delta}}$,
as clearly we get an appropriate surjection $\pi$.

We first show that $X\inter V_\delta=V_{\bar{\delta}}$.
Given $n<\om$ and $\xi<\delta$,  let
\[ \varepsilon_n(\xi)=\sup(\Hull_{\Sigma_n}^{V_\alpha}(V_\xi\cup\{p,\delta\}
)\inter\delta). \]
By inaccessibility of $\delta$, we have $\varepsilon_n(\xi)<\delta$.
Note
$\varepsilon_n(\xi)$ is definable over $V_\alpha$ from the parameter 
$(\xi,p,\delta)$ (since $n$ is fixed).
Therefore $\varepsilon_n``\bar{\delta}\sub\bar{\delta}$,
which gives $X\inter V_\delta=V_{\bar{\delta}}$.

Now for elementarity. First note that for each $n<\om$ and $\xi<\delta$,
there is $\eta<\delta$ such that for each
$\Sigma_n$ formula and $\vec{x}\in V_\xi^{<\om}$, if
$V_\alpha\sats\exists y\varphi(y,\vec{x},p,\delta)$
then there is $y\in V_\alpha$ with $V_\alpha\sats\varphi(y,\vec{x},p,\delta)$
and $y$ being $\Sigma_{n+3}$-definable from elements of 
$V_{\eta}\cup\{p,\delta\}$.
This is an easy consequence
of line (\ref{eqn:V_alpha_is_hull}) (using the usual trick of minimizing on 
ordinal parameters to get rid of them) and the inaccessibility of $\delta$.
Let $\eta_n(\xi)$ be the least $\eta$ that witnesses this for $\xi$. Note that 
$\eta_n``\bar{\delta}\sub\bar{\delta}$.
But then $X\elem V_\alpha$ and $X\inter V_\delta=V_{\bar{\delta}}$, as desired.
\end{proof}

\begin{dfn}
 Let $W\sats\RL$ be transitive. We say that $W$ is 
\emph{pwo-correct}
 iff for every $R\in W$, if $W\sats$``$R$ is a prewellorder''
 then $R$ is a prewellorder.
\end{dfn}

Note that if $W\sats\RL$ is transitive but not pwo-correct, then 
$W\not\sats\ZF$,
and in fact, letting $R\in W$ be a counterexample to pwo-correctness,
$W$ can define from $R$ a map from some initial segment of $R$ onto $\OR^W$.

\begin{dfn}
Let $\delta\in\OR$.
Let $M\sats\ZF$ be a transitive with $\delta\in\OR^M$.
Let $W$ be a transitive  and $j:V_\delta^M\to W$ be cofinal 
$\Sigma_1$-elementary.  Then $j$ is \emph{$M$-stable} iff either
\begin{enumerate}[label=--]
 \item $M\sats$``$\delta$ is inaccessible'', or
 \item $M\sats$``$\delta$ is singular'' and $j$
 is continuous at $\cof^M(\delta)$, or
\item  $M\sats$``$\delta$ is regular non-inaccessible''
and $j``R$ is cofinal in $j(R)$ for some/all $R\in\scot^M(\delta)$.\qedhere
\end{enumerate}
\end{dfn}

\begin{lem}\label{lem:V=HOD(V_delta)_cof(delta)>om}
Let $\delta\in\OR$.
Let $M$ be transitive with $M\sats\ZF+$``$V=\HOD(V_\delta)$''.
Let $W$ be  transitive and $j:V_\delta^M\to W$ be cofinal 
$\Sigma_1$-elementary and $M$-stable.
Then:
\begin{enumerate}
 \item\label{item:Los} {\L}o\'{s}' criterion holds for $\Ult_0(M,E_j)$.
\item\label{item:wfd} If $W$ is pwo-correct, $\cof^V(\delta)>\om$ and 
$\rg(j)\sub M$ then $\Ult_0(M,E_j)$ is 
wellfounded.
\end{enumerate}
\end{lem}

\begin{proof}
Part \ref{item:Los}:
Let $\varphi$ be $\Sigma_k$  and $\Omega\in\OR^M$
be such that $V_\Omega^M\elem_{k+2} M$;
so
\[ V_\Omega^M=\Hull_{\Sigma_2}^{V_\Omega^M}(V_\delta^M\cup\Omega).\]
Let
$\alpha<\delta$
and $f\in M$ with
$f:\left<V_\alpha^M\right>^{<\om}\to M$ and
$a\in\left<V^W_{j(\alpha)}\right>^{<\om}$
be such that for $E_a$-measure one many $u\in\left<V_\alpha^M\right>^{<\om}$,
we have
$V_\Omega^M\sats\exists y\ [\varphi(f(u),y)]$.

Work in $M$. For $u\in\left<V^M_\alpha\right>^{<\om}$, let $\beta_u$
be the least $\beta<\delta$
such that there are $z\in V_\beta$
and $y\in V_\Omega$ with 
$V_\Omega\sats\varphi(f(u),y)$
and $y$ is $\Sigma_2$-definable from ordinals and  $z$
over $V_\Omega$.
Given $\beta\leq\delta$, let $A_\beta$ be the set of all 
$u\in\left<V_\alpha\right>^{<\om}$ with $\beta_u\leq\beta$.

In  $V$, note $A_\delta\in E_a$, 
and if $\beta_0<\beta_1\leq\delta$ then 
$A_{\beta_0}\sub A_{\beta_1}$,
and $A_\delta=\bigcup_{\beta<\delta}A_\beta$.

\begin{clmthree}  $A_\beta\in E_a$ for some $\beta<\delta$.
\end{clmthree}
\begin{proof}
Suppose not. Then note that $\delta$ non-inaccessible in $M$.
Suppose $\delta$ is singular in $M$.
Let $\gamma=\cof^M(\delta)<\delta$ and $f\in M$ with $f:\gamma\to\delta$
cofinal.
Then
$A_\delta=\bigcup_{\xi<\gamma}A_{f(\xi)}$.
Now $\left<A_{f(\xi)}\right>_{\xi<\gamma}\in V^M_\delta$.
Now\[ 
j(A_\delta)=j(\bigcup_{\xi<\gamma}A_{f(\xi)})=\bigcup_{\xi<\gamma}j(A_{f(\xi)}
),\]
where the second equality holds by $M$-stability.
But $a\in j(A_\delta)$,
so there is $\xi<\gamma$ such that $a\in j(A_{f(\xi)})$,
and so $A_{f(\xi)}\in E_a$, a contradiction.

So $\delta$ is regular non-inaccessible in $M$.
We mimic the previous argument.
Let $R\in\scot^M(\delta)$. So $R\in V_\delta^M$.
Let $X$ be the field of $R$.
Let $\pi:X\to\delta$
be the corresponding norm, so $\pi$ is surjective
and $xRy$ iff $\pi(x)\leq\pi(y)$. We have
\[ \vec{A}=\left<A_{\pi(x)}\right>_{x\in X}\in V_\delta^M.\]
But $A_\delta=\bigcup_{x\in X}A_{\pi(x)}$, and since $j``X$ is 
cofinal in $j(R)$, we get like before $j(A_\delta)=\bigcup_{x\in 
X}j(A_{\pi(x)})$, so $a\in j(A_{\pi(x)})$ for some $x\in X$.
\end{proof}

Fix $\beta$ as in the claim.
Let $A'_\beta$ be the set of pairs $(u,z)$ such that
$u\in\left<V_\alpha^M\right>^{<\om}$
and $z\in V^M_\beta$ and $(u,z)$ are as above.
Since $a\in j(A_\beta)$, we have some $b$ with $(a,b)\in j(A'_\beta)$,
and  may assume $a\cup\{a\}\sub b\in\left<W\right>^{<\om}$,
by increasing $\beta$ if needed.

In $M$, for $z\in\left<A_\delta\right>^b$
with $(z^{ba},z)\in A'_\beta$, let $g(z)$ be the least
$y$ which is $\Sigma_2$-definable over $V_\Omega$ from 
ordinals and $z$, and such that
$V_\Omega\sats\varphi(f(z^{ba}),y)$
(minimize on the $\Sigma_2$ formula and ordinals
in order to minimize $y$).
Then note that for $E_b$-measure one many $z$, we have
$V^M_\Omega\sats\varphi( f^{ab}(z),g(z))$,
as desired.

Part \ref{item:wfd}: Suppose $\Ult_0(M,E)$
is illfounded; we argue much like 
in the proof of Theorem \ref{tm:cof(delta)>om_critical} for a contradiction.
We are now assuming that $W$ is pwo-correct, $\cof^V(\delta)>\om$ 
and $\rg(j)\sub M$.
For limits $\xi<\delta$ let $E_\xi=E_{j\rest V_\xi^M}$.
Like in the proof of Theorem \ref{tm:cof(delta)>om_critical},
as $\cof(\delta)>\om$,
we can fix $\xi<\delta$ such that $\Ult_0(M,E_\xi)$ is illfounded.
Since $\rg(j)\sub M$, note $E_\xi\in M$.

Work in $M$. Fix some $\Omega\in(\delta,\OR)$ and some large enough 
$k<\om$
with $V_\Omega\elem_{k+2} V$.
Then as before, $\Ult^{V_\Omega}_0(\Omega,E_\xi)$ is 
illfounded, where the notation
means we use all functions
in $V_\Omega$ to form the ultrapower. 
Now $V_\Omega=\Hull^{V_\Omega}_{\Sigma_2}(V_\delta\cup\Omega)$
and (because $k$ is large enough) $V_\Omega$ computes the ultrapower,
and its well/illfounded parts. So this reflects into 
$X=\Hull^{V_\Omega}(V_\delta\cup\{E_\xi\})$ (note we have 
dropped the parameters in $\Omega\cut\delta$), and $X\elem V_\Omega$.
Let $H$ be the transitive collapse of $X$ and $\bar{E}$ the collapsed
image of $E_\xi$. So
 $\mathscr{O}=\Ult^H_0(\OR^H,\bar{E})$ is illfounded.

For $\eta\leq\delta$, let $H_\eta=\Hull^H(V_\eta\cup\{\bar{E}\})$.
We do not claim  $H_\eta\elem H$.
But  $H_\delta=H=\bigcup_{\eta<\delta}H_\eta$.
Let $\mathscr{O}_\eta$ be the substructure
of $\mathscr{O}$ given by elements of the form $[a,f]$ where 
$a\in\left<S\right>^{<\om}$ and $S$
is the support of $\bar{E}$,
and $f\in H_\eta$.
So if $\eta_0<\eta_1$ then $\mathscr{O}_{\eta_0}$ is just
the restriction of $\mathscr{O}_{\eta_1}$ to its domain,
and $\mathscr{O}=\bigcup_{\eta<\delta}\mathscr{O}_\eta$.
As $\mathscr{O}$ is illfounded, we can fix
$\eta\in(\xi,\delta)$ such that $\mathscr{O}_\eta$ is illfounded;
this is again like 
in the proof of Theorem 
\ref{tm:cof(delta)>om_critical}.

Let $\gamma$ be the ordertype of $\OR\inter 
H_\eta$ and $\pi:\gamma\to\OR\inter H_\eta$ the uncollapse map,
so $\eta\leq\crit(\pi)$.
Given $f:\left<V_\xi\right>^{<\om}\to\OR$ with $f\in H_\eta$,
note $\rg(f)\sub\rg(\pi)$. Let 
$\bar{f}:\left<V_\xi\right>^{<\om}\to\gamma$
be the natural collapse (so $\pi\com\bar{f}=f$). We have
a surjection $\sigma:V_\eta\to\gamma$. Given $f$ as above, let $f'$ be the 
corresponding collapse; that is, $f':\left<V_\xi\right>^{<\om}\to V_{\eta+1}$,
\[ f'(u)=\{x\in V_\eta\bigm|\sigma(x)=\bar{f}(u)\}.\]
Let
$\mathscr{F}=\{f'\bigm|f\in H_\eta\text{ and 
}f:\left<V_\xi\right>^{<\om}\to\OR\}$.
Let $\leq_\sigma$ be the prewellorder of $V_\eta$ induced by $\sigma$.
Clearly then $U=\Ult_0^{\mathscr{F}}({\leq_\sigma},E)$ is illfounded,
as it is in fact isomorphic to $\mathscr{O}_\eta$.

But now back in $V$, by the $\Sigma_1$-elementarity of $j$, 
$W\sats$``$j({\leq_\sigma})$ is a 
prewellorder of $V_{j(\eta)}$'',
so by pwo-correctness, this is truly a prewellorder.
But  we can absorb  $U$ into $j({\leq_\sigma})$ as usual, by mapping
$[a,f']^{\mathscr{F}}_E$ to $j(f')(a)$.
So $j({\leq_\sigma})$ is illfounded,  a contradiction.
\end{proof}

\begin{lem}\label{lem:rank-to-rank_lemma}
Suppose  $V=\HOD(V_\delta)$ where 
$\cof(\delta)>\om$.
Let $j\in\mathscr{E}_1(V_\delta)$.
 Then for all sufficiently large $m<\om$,
 {\L}o\'{s}' criterion holds for $\Ult(V,E_{j_m})$, and this ultrapower is 
wellfounded.
\end{lem}
\begin{proof}
 For all sufficiently large $m$, $j_m:V_\delta\to V_\delta$ is $V$-stable,
 by \cite[Theorem 5.6***]{cumulative_periodicity}. So we can apply Lemma 
\ref{lem:V=HOD(V_delta)_cof(delta)>om}.
\end{proof}

\begin{dfn}
 Given $M\sats\RL$ and $n<\om$, let $C^M_n$ be the $M$-class
 of all $\alpha\in\OR^M$ with $V_\alpha^M\elem_n M$ (we do not assume that 
$M$ is wellfounded here).
\end{dfn}

Note that $C^M_0=\OR^M$, and $M\sats\ZF$ iff $C^M_n$ is a proper $M$-class
for all $n<\om$.

\begin{lem}\label{lem:ZF_equiv}
 Let $M,N\sats\RL$ and $j:M\to N$ be $\in$-cofinal $\Sigma_1$-elementary.
 Then:
 \begin{enumerate}
  \item\label{item:j[C^M_n]} If $C^M_n$ is a proper $M$-class then 
$j``C^M_n\sub C^N_n$ and $j$ is $\Sigma_{n+1}$-elementary.
\item\label{item:ZF_equiv} $M\sats\ZF$ iff $N\sats\ZF$.
 \end{enumerate}
 \end{lem}
\begin{proof}
Part \ref{item:j[C^M_n]}: This is by induction on $n$.
For $n=0$ it is by assumption. Suppose it holds at $n$
and $C^M_{n+1}$ is a proper $M$-class. 
Let $\alpha\in C^M_{n+1}$ and  $\varphi$ be $\Pi_{n}$
with
\[ N\sats\exists x\in V_{j(\alpha)}[\exists w\varphi(w,x)\text{ but 
}V_{j(\alpha)}\sats\neg\exists w\varphi(w,x)].\]
Fix $\beta\in C^M_n$ such that $\alpha<\beta$ and there
are $(x,w)$ witnessing this statement with $w\in j(V_\beta^M)$
(and $x\in j(V_\alpha^M)$). By induction,
$j(\beta)\in C^N_n$. Therefore
\[ N\sats\exists x\in j(V_{\alpha}^M)\ [j(V_\beta^M)\sats\exists 
w\varphi(w,x)\text{ but }j(V_\alpha^M)\sats\neg\exists w\varphi(w,x)].\]

Note this is a $\Sigma_1$ statement  the parameters 
$j(V_\alpha^M),j(V_\beta^M)$,
so it reflects back to 
$M,V_\alpha,V_\beta$, which contradicts that
 $\alpha\in C^M_{n+1}$ and $\beta\in C^M_n$.

 The fact that $j$ is $\Sigma_{n+1}$-elementary
 is now an easy consequence of the fact that $C^M_{n+1}$ is a proper $M$-class
and $j``C^M_{n+1}\sub C^N_{n+1}$.

Part \ref{item:ZF_equiv}: Since $M\sats\ZF$ iff $C^M_n$ is a proper $M$-class
for all $n<\om$ (and hence likewise for $N$),
the implication $M\sats\ZF\Rightarrow N\sats\ZF$ follows immediately
from part \ref{item:j[C^M_n]}. Conversely, suppose $M\not\sats\ZF$,
and let $n$ be least such that $C^M_{n+1}$ is not a proper $M$-class
(note $C^M_0=\OR^M$). We claim that $C^N_{n+1}$
is not a proper $N$-class. For by part \ref{item:j[C^M_n]}, 
$j``C^M_n\sub C^N_n$ and $j$ is $\Sigma_{n+1}$-elementary.
Let
$\alpha\in\OR^M$ with $\alpha>\sup(C^M_{n+1})$.
Suppose there is $\gamma\in C^N_{n+1}$
with $\gamma>j(\alpha)$.
Let $\beta\in\OR^M$ be least such that
$j(\beta)\geq\gamma$.
We claim that  
$\beta\in C^M_{i+1}$
for all $i\leq n$, a contradiction.
We show this by induction on $i$.
So suppose that $\beta\in C^M_i$,
and let $\varphi$ be $\Pi_{i}$
and $x\in V_\beta^M$ and suppose that $w\in M$
and $M\sats\varphi(x,w)$ but $V_\beta^M\sats\neg\exists w\varphi(w,x)$.
By part \ref{item:j[C^M_n]},
$j(\beta)\in C^N_i$,
and since $j$ is $\Sigma_{n+1}$-elementary,
$N\sats\varphi(j(w),j(x))$ but $V_{j(\beta)}^N\sats\neg\exists 
w\varphi(w,j(x))$. So $N\sats\exists w\varphi(w,j(x))$,
and therefore $V_\gamma^N\sats\exists w\varphi(w,j(x))$.
But $V_\gamma^N\elem_i V_{j(\beta)}^N$,
and therefore $V_{j(\beta)}^N\sats\exists w\varphi(w,j(x))$,
a contradiction.
\end{proof}

We next  deduce the main result of this section.
This is easy except for part 
\ref{item:no_j_deltabar_to_delta}\ref{item:maximal_extension_of_j},
for which we 
assume familiarity with the fine 
structural
$\rSigma_n$ hierarchy,
basically as defined in 
\cite{scales_in_LR} for $L(\RR)$,
but with $V_{\bar{\delta}}$ or $V_\delta$ replacing $\RR$;
we  also assume familiarity with fine structural
ultrapowers $\Ult_n$. We first give a summary of the definitions and facts we 
need;
their proofs are as in the standard fine structure literature.

\begin{rem}Work with 
$M=\J_\alpha(V_\delta)$.
Define $\rSigma_0^{M}=\Sigma_0^{M}$
in the language of set theory together with a symbol for $V_\delta$,
and $\rSigma_1=\Sigma_1^{M}$ in this language.
Given 
$\rSigma_{n+1}^{M}$, 
 $\bfrSigma_{n+1}^M$ denotes the associated
 boldface class, allowing arbitrary parameters in $M$.
 Define
the \emph{$(n+1)$st projectum} $\rho_{n+1}^{M}$ 
as the
least ordinal $\rho$
such that there is $X\sub\J_\rho(V_\delta)$
such that $X$ is $\bfrSigma_{n+1}^{M}$-definable
but $X\notin M$.

If $\rho=\rho_{n+1}^{M}>0$,
then 
set $T_{n+1}=$ the set of all theories 
\[ t=\Th_{\rSigma_{n+1}}^{M}(V_\delta\cup\gamma\cup\{x\}) 
\]
with $\gamma<\rho$ and $x\in M$;
hereby
\[ t=\{(\varphi,(\vec{z},\vec{\beta},x))\bigm|\varphi\text{ is }\rSigma_{n+1},\ 
\vec{z}\in V_\delta^{<\om},\ \vec{\beta}\in\gamma^{<\om}\text{ and }
M\sats\varphi(\vec{z},\vec{\beta},x)\}.\]
If instead $\rho=0$ (that is, we have some $X$ as above with $X\sub V_\delta$),
then set $T_{n+1}=$ the set of all theories
\[ t=\Th_{\rSigma_{n+1}}^{M}(V_\gamma\cup\{x\}) \]
with $\gamma<\delta$ and $x\in M$.
Note that in either case, each such $t\in M$.

Then $\rSigma_{n+2}^{M}$ is the class of relations
$\varphi$ defined over $M$ 
in the form
\[ \varphi(\vec{x})\iff\exists t\in T_{n+1}\ [\psi(t,\vec{x})],\]
where $\psi$ is $\rSigma_1$. The $\rSigma_n^{M}$
relations are cofinal in the $\Sigma_n^{M}$ relations.
One also defines the \emph{$(n+1)$st 
standard parameter} $p_{n+1}^{M}$
as the least $p\in[\delta+\alpha]^{<\om}$
(with respect to the ordering $p<^*q$ iff $\max(p\Delta q)\in q$)
such that  the theory
\[\Th_{\rSigma_{n+1}}^{M}(V_\delta\cup\rho\cup\{p,\pvec_n^M\}
)\notin M, \]
where $\rho=\rho_{n+1}^{M}$ and $\pvec_n=(p_1,\ldots,p_n)$.
One can then prove \emph{$(n+1)$-soundness},
which implies in particular that for every $x\in M$ there is an $\rSigma_{n+1}$
formula $\psi$ and $z\in V_\delta$ and $\alpha<\max(\rho,1)$ such that
$x$ is the unique $x'\in M$
with $M\sats\psi(x',\vec{p}_{n+1}^{M},z,\alpha)$.

One can also prove $\bfrSigma_n$-uniformization modulo $V_\delta$.
That is, given  $R\sub X\cross Y$,  a partial
function
$g:_{\mathrm{p}}X\cross V_\delta\to Y$
is a $V_\delta$-\emph{uniformization of $R$} if
$(x,g(x,z))\in R$ for each $(x,z)\in\dom(g)$,
and for each $(x,y)\in R$, there is $z$ with $(x,z)\in\dom(g)$.
Then $\bfrSigma_n$-uniformization modulo $V_\delta$
says that if $n=0$ then for each $R\in M$,
there is a $V_\delta$-uniformization of $R$  in
$M$,
and if $n>0$ then for each 
$\bfrSigma_n^{M}$-definable relation,
there is an $\bfrSigma_n^{M}$-definable 
$V_\delta$-uniformization
of $R$.

Given an extender over $V_\delta$,
the fine structural ultrapower $\Ult_n(M,E)$ is 
just the 
ultrapower
formed using all functions $f$ which are in $M$,
or if $n>0$, functions
$f$ which are $\bfrSigma_n^{M}$-definable.
If {\L}o\'{s}' theorem holds with respect to $\rSigma_n$
formulas and such functions, then the ultrapower embedding
is $\rSigma_{n+1}$-elementary.
\end{rem}

Part \ref{item:no_j_deltabar_to_delta}\ref{item:maximal_extension_of_j} of 
the theorem below says that $j$ extends to a map which exhibits
 essentially as much elementarity as possible,
because any further
elementarity would imply that $j$ is in fact $L(V_{\bar{\delta}})$-stable,
giving a contradiction.
The author does not know whether such a $j:V_{\bar{\delta}}\to V_\delta$ can 
actually be in $L(V_\delta)$, however.

\begin{tm}\label{tm:V=L(V_delta)_cof(delta)>om}
 \tu{(}$\ZF$\tu{)} Assume $V=L(V_\delta)$ where
 $\cof(\delta)>\om$. Then:
 \begin{enumerate}
\item\label{item:no_j_in_L(V_delta)} $\mathscr{E}_1(V_\delta)=\emptyset$.
\item\label{item:no_j_deltabar_to_delta}For $\bar{\delta}<\delta$,
if $j:V_{\bar{\delta}}\to V_\delta$
is cofinal $\Sigma_1$-elementary then:
\begin{enumerate}[label=\tu{(}\alph*\tu{)}]
\item\label{item:j_not_stable} $j$ is not $L(V_{\bar{\delta}})$-stable.
\item\label{item:when_j_fully_elem} $j$ is fully elementary iff 
$V_{\bar{\delta}}\sats\ZF$ iff $V_\delta\sats\ZF$.
\item\label{item:maximal_extension_of_j} Let
$(\bar{\xi},\bar{n},\bar{\eta})$ be lex-least such that
$\bar{\eta}<\bar{\delta}$ and there is an 
$\bfrSigma_{n+1}^{\J_{\bar{\xi}}(V_{\bar{\delta}})}$
cofinal function $\bar{f}:V_{\bar{\eta}}\to\bar{\delta}$,
and $(\xi,n,\eta)$ likewise for $\delta$.
Then $\bar{n}=n$ and there is a unique 
\[ k:\J_{\bar{\xi}}(V_{\bar{\delta}})\to\J_\xi(V_\delta) \]
such that (i) $k$ is $\rSigma_{n+1}$-elementary, (ii) $j\sub k$, and (iii) if 
$0<\bar{\xi}$
then $k(\bar{\delta})=\delta$. Moreover, $k$ is not 
$\rSigma_{k+2}$-elementary.
\item\label{item:j_not_lower_than_xi,n} $j\notin \J_\xi(V_\delta)$, and if 
$n>0$ then
$j$ is not $\bfrSigma_n^{\J_\xi(V_\delta)}$-definable.
\end{enumerate}
\end{enumerate}
\end{tm}
\begin{proof}
Part \ref{item:no_j_in_L(V_delta)}: Suppose $j\in\mathscr{E}_1(V_\delta)$. By 
Lemma \ref{lem:rank-to-rank_lemma}, we may assume
 {\L}o\'{s}' theorem and wellfoundedness for $\Ult(V,E_j)$.
But then as in Remark
\ref{rem:no_R_in_L(X)},
we get $i_{E_j}:V\to V$ is definable from $E_j$,
a contradiction.

Part \ref{item:no_j_deltabar_to_delta}\ref{item:j_not_stable}:
Suppose otherwise. By
Lemma \ref{lem:V=HOD(V_delta)_cof(delta)>om},
$\Sigma_0$-Los' criterion holds for
$\Ult(L(V_{\bar{\delta}}),E_j)$, so
the ultrapower map is $\Sigma_1$-elementary and hence
fully elementary.
The lemma does not give directly that the ultrapower is wellfounded,
because $\rg(j)\not\sub L(V_{\bar{\delta}})$,
but we can argue similarly to see that it is,
because 
$V_{\bar{\delta}}\sub L(V_{\bar{\delta}})$.
That is, if the ultrapower is illfounded, then since 
$\cof(\bar{\delta})=\cof(\delta)>\om$,
we can find some limit $\eta<\bar{\delta}$
and some $\chi\in\OR$
such that $\Ult_{\mathscr{F}}(\chi,E_{j\rest V_\eta})$
is illfounded, where $\mathscr{F}$ is the set of all functions
$f:V_\eta\to\chi$ such that $f$ is definable
from elements of $V_\eta$
and ordinals over $\J_\chi(V_{\bar{\delta}})$.
We can then build a sufficiently elementary hull
$\bar{\mathscr{F}}$ of this set of functions,
and a surjection $\pi:V_\eta\to\bar{\mathscr{F}}$,
such that $\Ult_{\bar{\mathscr{F}}}(\chi,E_{j\rest V_\eta})$
is also illfounded. But then we get an isomorphic
copy of $\bar{\mathscr{F}}$ in $V_{\bar{\delta}}$,
and reach a contradiction.

So the ultrapower map
$j^+:L(V_{\bar{\delta}})\to L(V_\delta)=V$ is elementary. But $j^+$ is 
definable over $V$ from $j$,
contradicting Suzuki \cite[Theorem 3.1]{suzuki_no_def_j}.

Part \ref{item:no_j_deltabar_to_delta}\ref{item:when_j_fully_elem}: By Lemma 
\ref{lem:ZF_equiv},
$V_{\bar{\delta}}\sats\ZF$ iff $V_\delta\sats\ZF$. If they model
$\ZF$ then $j$ is fully elementary by \cite{gaifman}.
Now suppose that $j$ is fully elementary but $\ZF$ fails,
for a contradiction. Suppose there is $\bar{\gamma}<\bar{\delta}$
and some cofinal function $\bar{f}:\bar{\gamma}\to\bar{\delta}$
which is definable from parameter $\bar{p}$ over $V_{\bar{\delta}}$.
Then defining $f:\gamma\to\delta$ over $V_\delta$ in the same way from 
$p=j(\bar{p})$,  
$f$ is cofinal in $\delta$. But note also that $j\com\bar{f}=f\com 
j\rest\bar{\gamma}$,
and therefore $j$ is continuous at $\bar{\gamma}$.
It easily follows that $j$ is $L(V_{\bar{\delta}})$-stable, contradiction.
It is similar in the other case.

Part  \ref{item:no_j_deltabar_to_delta}\ref{item:maximal_extension_of_j}:
Let $M=\J_{\bar{\xi}}(V_{\bar{\delta}})$, 
$U=\Ult_{\bar{n}}(M,,E_j)$
and $k$ the ultrapower map.

Almost like before,
$U$ is wellfounded. (Every element of $M$
is definable over $M$ from some element
of $V_{\bar{\delta}}$, because of the minimality of $\bar{\xi}$.
For $\eta<\bar{\delta}$, let $\mathscr{F}_\eta$ be the set of functions
used in forming the ultrapower which are definable over 
$M$ from some element of $V_\eta$.
Then since $\cof(\bar{\delta})=\cof(\delta)>\om$, we can find
$\eta<\bar{\delta}$ such that 
$\Ult_{\mathscr{F}_\eta}(M,E_{j\rest V_\eta})$
is illfounded. Now take a hull and proceed like before for a contradiction.)

{\L}o\'{s}' 
theorem goes through for all $\rSigma_{\bar{n}}$ formulas
(in arbitrary parameters of the ultrapower).
For let $\varphi$ be an $\rSigma_{\bar{n}}$ formula and
$f$ an $\bfrSigma_{\bar{n}}^{M}$-definable
function (with $f\in M$ if ${\bar{n}}=0$).
Let $X\in V_{\bar{\delta}}$ be some measure one set,
and suppose
\[ \all u\in X\ \Big[M\sats\exists w\ \varphi(w,f(u))\Big].\]
Let $g:X\cross V_\delta\to M$
be an $\bfrSigma_{\bar{n}}^{M}$-definable
$V_\delta$-uniformization of the relation $[u\in X\wedge \varphi(f(u),w)^M]$. 
Note then that by the minimality of $(\bar{\xi},{\bar{n}})$,
there is $\eta<\bar{\delta}$ such that for all $u\in X$,
there is $z\in V_\eta$ with $(u,z)\in\dom(g)$.
From here we can proceed as before.

It follows that $k$ is $\rSigma_{n+1}$-elementary.
And $k$ is continuous at $\bar{\delta}$ also by choice of $(\bar{\xi},\bar{n})$
(where if $\bar{\xi}=0$, this means that $\delta=\OR^U$).
But now let $\bar{f}$ witness the choice of $(\bar{\xi},\bar{n},\bar{\eta})$,
and take this with the ordertype of $\rg(\bar{f})$ minimal.
Let $\bar{<}$ be the induced prewellorder on $V_{\bar{\eta}}$.
So $\bar{f}$ is monotone increasing from $\bar{<}$
to $\in$ (and $\bar{<}\in V_{\bar{\delta}}$).
Fix a parameter $\bar{p}\in\J_{\bar{\xi}}(V_{\bar{\delta}})$
and an $\rSigma_{n+1}$ formula $\psi$
such that $\bar{f}(x)=y$ iff 
$\J_{\bar{\xi}}(V_{\bar{\delta}})\sats\psi(\bar{p},x,y)$,
and also with $\bar{p}$ directly encoding $\bar{<}$.
Note that the statement ``$\bar{f}:V_{\bar{\eta}}\to\bar{\delta}$''
is just ``for all $x\in V_{\bar{\eta}}\ \exists y\ [\psi(\bar{p},x,y)]$'',
which is $\neg\rSigma_{n+2}(\{\bar{p}\})$. 
And ``$\rg(\bar{f})$ is cofinal in $\bar{\delta}$''
is also  expressible at this level of complexity.
And note that the statement
``$\bar{f}:_{\mathrm{p}}(V_{\bar{\eta}},\bar{<})\to(\bar{\delta},\in)$
is a partial function which is monotone increasing on its domain'' is 
$\neg\rSigma_{n+1}(\{\bar{p}\})$.

Let $f$ be defined over $U$ from $p=k(\bar{p})$
and $\psi$, with domain $V_{\eta'}=k(V_{\bar{\eta}})=j(V_{\bar{\eta}})$.
Let ${<}=k(\bar{<})=j(\bar{<})$.
Since $k$ is $\rSigma_{n+1}$-elementary
and $k(\bar{\delta})=\delta$ (or $\bar{\xi}=0$ and $U=V_\delta$),
then $f:_{\mathrm{p}}(V_\eta,<)\to(\delta,\in)$
is a partial function which is monotone increasing on its domain,
and also $k\com\bar{f}=f\com k\rest V_{\bar{\eta}}$,
and therefore $f\rest k``V_{\bar{\eta}}$ is cofinal in $\delta$.
So letting $\J_{\xi'}(V_\delta)=U$, it follows
that $(\xi',\bar{n},\eta')\leq_{\lex}(\xi,n,\eta)$.
Moreover, $(\xi',\bar{n})=(\xi,n)$ by 
 $\rSigma_{n+1}$-elementarity.
 And note that by the elementarity of $j:V_{\bar{\delta}}\to V_\delta$,
 there is no $\eta''<\eta'$ and function $g:V_{\eta''}\to V_\eta$
 which is cofinal in $<$, and so $\eta'=\eta$.
 
Since $j$ is not $L(V_{\bar{\delta}})$-stable,
 it is not true that $f:(V_\eta,<)\to(\delta,\in)$
 is total, cofinal and monotone-increasing, and therefore
 $k$ is not $\rSigma_{n+2}$-elementary.
 
 Part \ref{item:j_not_lower_than_xi,n}: Since $\bar{\delta}<\delta$ and 
$j``\bar{\delta}$ is cofinal in $\delta$, this is just by choice of 
$(\xi,n)$.
\end{proof}

\section{Admissible 
$L_\kappa(V_\delta)$ and countable cofinality}\label{sec:admissible}

\begin{lem}\label{lem:V=L(V_delta)_cof(delta)=om}
\tu{(}$\ZF$\tu{)} Assume $V=L(V_\delta)$ where $\delta\in\Lim$ and 
$\cof(\delta)=\om$.
Suppose $\bar{\delta}\leq\delta$
and
$j:V_{\bar{\delta}}\to V_\delta$
is $\Sigma_1$-elementary
and $\in$-cofinal.\footnote{Note that if $\bar{\delta}=\delta$
then the hypothesis that $\cof(\delta)=\om$ is redundant,
by Theorem \ref{tm:V=L(V_delta)_cof(delta)>om}.}
Then:
\begin{enumerate}
\item\label{item:cof(delta)=om} $j$ 
is fully elementary.\footnote{Note that Theorem 
\ref{tm:V=L(V_delta)_cof(delta)>om}
does not allow us to conclude that $V_\delta\sats\ZF$,
because $\cof(\delta)=\om$.}
\item\label{item:cof(delta-bar)_in_inner} 
$L(V_{\bar{\delta}})\sats$``$\bar{\delta}$
is inaccessible or $\cof(\bar{\delta})=\om$'', so $j$ is 
$L(V_{\bar{\delta}})$-stable.
 \item\label{item:Sigma_0_Los_crit_for_J_alpha} For all $\alpha\leq\OR$,
 $\Sigma_0$-{\L}o\'{s}' criterion
 holds for $\Ult_0(\J_\alpha(V_{\bar{\delta}}),E_j)$.
\item\label{item:Ult_illfd}  $\Ult_0(L(V_{\bar{\delta}}),E_j)$ is illfounded.
\end{enumerate}
\end{lem}
\begin{proof}[Proof Sketch]
We write $\J_\alpha$ for $\J_\alpha(V_\delta)$.

Part \ref{item:cof(delta)=om}:
By  \cite[Theorem 
5.6***]{cumulative_periodicity} if $\bar{\delta}=\delta$,
and otherwise,  
the proof is a straightforward adaptation.
(Show by induction on $n\in[1,\om)$
that $j$ is $\Sigma_n$-elementary
and that for each $\alpha<\bar{\delta}$
we have
\[ 
j(\Th_{\Sigma_n}^{V_{\bar{\delta}}}(V_\alpha))=\Th_{\Sigma_n}^{V_{\delta}}(V_{
j(\alpha)}).\]
The proof of these facts is just like in \cite{cumulative_periodicity}.)

Part \ref{item:cof(delta-bar)_in_inner}:
A routine consequence of the fact that $\cof(\bar{\delta})=\cof(\delta)=\om$ in 
$V$.

Part \ref{item:Sigma_0_Los_crit_for_J_alpha}: This is much as 
in the proof of Lemma \ref{lem:V=HOD(V_delta)_cof(delta)>om}, using that 
\[ \J_\beta=\Hull_{\Sigma_1}^{\J_\beta}
(V_{\bar{\delta}}\cup(\bar{\delta}+\beta))\]
for all $\beta\leq\alpha$.
Suppose for example that $\alpha=\beta+1$.
Let $f,h\in\J_\alpha$ with $f,h:\left<V_{\bar{\delta}}\right>^{<\om}\to V$.
Suppose that for $E_a$-measure one many $u$,
we have
\[ \J_\alpha\sats\exists y\in h(u)\ [\varphi(f(u),y)],\]
where $\varphi$ is $\Sigma_0$. We have $m<\om$
such that $f,g\in\Ss_m(\J_\beta)$, where $\Ss$ denotes
Jensen's $\Ss$-operator (so $\J_\alpha=\bigcup_{k<\om}\Ss_k(\J_\beta)$).
Fix a surjection 
\[ \pi:(V_{\bar{\delta}}\cross\beta^{<\om})\to\Ss_m(\J_\beta) \]
with $\pi\in\J_\alpha$. Then arguing as before,
using that $\cof(\bar{\delta})=\om$,
we can find $\xi<{\bar{\delta}}$ such that for $E_a$-measure one many
$u$, there is $y\in\pi``(V_\xi\cross\beta^{<\om})$
such that
\begin{equation}\label{eqn:J_alpha_sats_varphi} \J_\alpha\sats y\in h(u)\ \&\ 
\varphi(f(u),y).\end{equation}
Now for pairs $(u,v)\in\left<V_{\bar{\delta}}\right>^{<\om}\cross V_\xi$,
let $g'(u,v)$ be the least $y\in\pi``(\{v\}\cross\beta^{<\om})$
such that line (\ref{eqn:J_alpha_sats_varphi}) holds,
if there is such a $y$. Then we find an appropriate
index $b$ and convert $g'$ into a function $g$,
with $(b,g)$ witnessing $\Sigma_0$-{\L}o\'{s}' criterion, like before.

Part \ref{item:Ult_illfd}: By part 
\ref{item:Sigma_0_Los_crit_for_J_alpha}, we would otherwise get
$i^{L(V_{\bar{\delta}}),0}_E:L(V_{\bar{\delta}})\to V$ elementary, 
contradicting Suzuki \cite[Theorem 
3.1]{suzuki_no_def_j}.
\end{proof}

\begin{rem}
Let $M$ be a transitive set. Recall that $M$ is \emph{admissible}
iff $M$ satisfies Pairing, Infinity,
$\Sigma_0$-Separation,
and whenever $d,p\in M$ and $\varphi$ is a 
$\Sigma_1$ formula and
$M\sats\all x\in d\ \exists y\ \varphi(x,y,p)$
then there is $e\in M$ such that
$M\sats\all x\in d\ \exists y\in e\ \varphi(x,y,p)$.
\end{rem}
\begin{dfn}\label{dfn:kappa_X}
 Given a transitive set $X$, let $\kappa_X$ denote
 the least $\kappa\in\OR$ such that $\J_\kappa(X)$ is admissible.
\end{dfn}

Recall the notation $\wfp$ and $\illfp$ from \S\ref{subsec:notation}.
A well-known fact is:
\begin{fact}\label{fact:wfp_admissible}
\tu{(}$\ZF$\tu{)} Let $M$ be an extensional structure in the language of set 
theory,
 let $X\in\wfp(M)$
 \tu{(}and we assume $M$ is transitive below $X$\tu{)}.
Suppose $M\sats$``$V=L(X)$''
 \tu{(}but $M$ might not satisfy $\ZF$\tu{)}. If $M$ is 
illfounded then $\J_{\kappa_X}(X)\sub M$.
\end{fact}
\begin{proof}
Let $\lambda=\OR\inter\wfp(M)$.
 Because $\lambda\psub\OR^M$ but $\lambda\notin M$,
 and $M\sats$``$V=L(X)$'', and hence, $M\sats$``I am rudimentarily closed'',
 it is easy to see that $\lambda$ is closed under ordinal addition and 
multiplication. Moreover, it is easy to see that $\J_\lambda(X)\sub M$,
and hence, $\J_\lambda(X)\sub\wfp(M)$.

Now suppose that $\lambda<\kappa_X$. Then we can fix a $\Sigma_1$ formula 
$\varphi$
 and $d,p\in\J_\lambda(X)$ such that $\lambda$ is the least $\lambda'$ such that
 \[ \J_{\lambda'}(X)\sats\all x\in d\ \exists y\ \varphi(x,y,p).\]
Note then that for all $\alpha\in\OR^M\cut\lambda$,
\[ M\sats\text{``}\J_\alpha(X)\sats\all x\in d\ \exists y\ 
\varphi(x,y,p)\text{''}.\]
But then for such $\alpha$,
\[ M\sats\text{``}\{\alpha\in\OR\bigm|\J_\alpha(X)\sats\neg\all x\in d\ \exists 
y\ \varphi(x,y,p)\}\in\J_{\alpha+1}(X)\text{''}.\]
But note that this set is exactly $\lambda$, so $\lambda\in M$, a contradiction.
\end{proof}

\begin{fact}\label{fact:always_Hull_below_admissible}
\tu{(}$\ZF$\tu{)} Let $X$ be transitive. Then for every $\alpha\leq\kappa_X$,
 we have
 \[ 
\J_\alpha(X)=\Hull_{\Sigma_1}^{\J_\alpha(X)}
(X\cup\{X\}).\]
Therefore \tu{(}i\tu{)}
$\pow(X)\inter\J_{\alpha+1}(X)\not\sub\J_\alpha(X)$
and \tu{(}ii\tu{)} for every $x\in\J_\alpha(X)$ there is a surjection
$\pi:X\to x$ with $\pi\in\J_\alpha(X)$.
\end{fact}
\begin{proof}
Let $H=\Hull_1^{\J_\alpha(X)}(X\cup\{X\})$
and $\beta=\sup(H\inter\OR)$. Note then that
$H=\J_\beta(X)$.
So it suffices to see that $\beta=\alpha$, so suppose $\beta<\alpha$.
Then $\J_\beta(X)$ is inadmissible. So let $p,d\in\J_\beta(X)$
and $\varphi$ be $\Sigma_1$, such that $\beta$ is least such that
\[ \J_\beta(X)\sats\all x\in d\ \exists y\ \varphi(x,y,p).\]
Then
\[ \J_\alpha(X)\sats\exists\beta'\in\OR\ [\J_{\beta'}(X)\sats\all 
x\in d\ \exists y\ \varphi(x,y,p)].\]
But $\beta$ is the least such $\beta'$, and since
$p,d\in H$, it follows that $\beta\in H$, a contradiction.

Part (i) of the ``therefore'' clause now follows by a standard diagonalization.
For part (ii), if $\alpha=\beta+1$, use that 
$\J_\alpha=\bigcup_{n<\om}\Ss_n(\J_\beta)$
(where $\Ss_n$ is the $n$th iterate of Jensen's $\Ss$-operator)
and for each $n\in[1,\om)$, 
\[ 
\Ss_n(\J_\beta)=\Hull_{\Sigma_1}^{\Ss_n(\J_\beta)}
(X\cup\{X,\beta\}).\qedhere\]
\end{proof}

We now prove the promised strengthening of Theorem 
\ref{tm:j_not_amenably_Sigma_1} (note that if $\bar{\delta}=\delta$
then the hypothesis that $\cof(\delta)=\om$ is redundant, by Theorem 
\ref{tm:V=L(V_delta)_cof(delta)>om}):
\begin{tm}\label{tm:no_j_in_adm} \tu{(}$\ZF$\tu{)} Let $\delta\in\Lim$ with 
$\cof(\delta)=\om$ and 
$j:V_{\bar{\delta}}\to 
V_\delta$ be
$\Sigma_1$-elementary and $\in$-cofinal. Let 
$\theta=\kappa_{V_\delta}$  \tu{(}see \ref{dfn:kappa_X}\tu{)}. Then 
$j\notin\J_\theta(V_\delta)$.
In fact,  $j$ is not $\bfSigma_1^{\J_\theta(V_\delta)}$,
and  not $\bfPi_1^{\J_\theta(V_\delta)}$.
\end{tm}
\begin{proof}
 We write $\J_\alpha$ for $\J_\alpha(V_\delta)$.
Suppose first  $j\notin\J_\theta$,
and we deduce the rest. Let $\varphi$ be $\Sigma_1$ 
and $p\in\J_\theta$,
and suppose  for all $x\in V_{\bar{\delta}}$ and $y\in V_\delta$, we have
$j(x)=y$ iff $\J_\theta\sats\varphi(x,y,p)$.
Then note that
\[ \J_\theta\sats\all x\in V_{\bar{\delta}}\ \exists\alpha\in\OR\ 
[\J_\alpha\sats\exists y\in V_\delta\ \varphi(x,y,p)], \]
and so by admissibility, there is $\lambda<\theta$ such that
\[ \J_\lambda\sats\all x\in V_{\bar{\delta}}\ \exists y\in V_\delta 
\ \varphi(x,y,p).\]
But then for $x\in V_{\bar{\delta}}$ and $y\in V_\delta$, we have
$j(x)=y$ iff $\J_\lambda\sats\varphi(x,y,p)$,
so $j\in\J_\theta$, contradiction.

Now suppose that for all $x,y$, we have
$j(x)=y$ iff $\J_\theta\sats\neg\varphi(x,y,p)$.
Note for each $x\in V_{\bar{\delta}}$, letting $y=j(x)$,
\[ \J_\theta\sats\all y'\in V_\delta\cut\{y\}\ \varphi(x,y',p),\]
and so by admissibility, there is (a least) $\alpha_x<\theta$ such that
\[ \J_{\alpha_x}\sats\all y'\in V_\delta\cut\{y\}\ \varphi(x,y',p).\]
Then
$\J_\theta\sats\all x\in V_{\bar{\delta}}\ 
\exists\alpha\in\OR\ [\J_\alpha\sats\exists y\in V_\delta\ \all 
y'\in V_\delta\cut\{y\}\ \varphi(x,y',p)]$,
but then by admissibility, we get $\sup_{x\in V_{\bar{\delta}}}\alpha_x<\theta$,
but then $j\in\J_\theta$, a contradiction.

So suppose $j\in\J_\theta$. Then
$\cof^{L(V_\delta)}(\delta)=\cof^{L(V_\delta)}(\bar{\delta})=\om$.
For if $\bar{\delta}=\delta$ then 
$\cof(\delta)^{L(V_\delta)}=\om$ by Theorem 
\ref{tm:V=L(V_delta)_cof(delta)>om},
and if $\bar{\delta}<\delta$ then since $j$ is cofinal, 
$\delta$ is not inaccessible in $L(V_\delta)$, but then since $V_\delta\sub 
L(V_\delta)$, $\cof^{L(V_\delta)}=\cof^V(\delta)=\om$.
So Lemma \ref{lem:V=L(V_delta)_cof(delta)=om},
applies.

We now give the full argument
assuming that $\bar{\delta}=\delta$, and then sketch the changes
for the case that $\bar{\delta}<\delta$.
Let $E=E_j$ and  $\alpha<\theta$ with $E\in\J_\alpha$.
Let $\kappa=\crit(E)$. 
Let $M=\J_{\alpha+\kappa+1}(V_\delta)$ and
$U=\Ult_0(M,E)$.
By Lemma \ref{lem:V=L(V_delta)_cof(delta)=om},
$\Sigma_0$-{\L}o\'{s}' criterion holds for $U$,
so $i^{M}_E$ is $\in$-cofinal
and $\Sigma_1$-elementary. Because $M\sats$``$V=L(V_\delta)$'',
therefore $U\sats$``$V=L(i^M_E(V_\delta))$''.

\begin{clmtwo} $i^M_E(V_\delta)=V_\delta$.\end{clmtwo}
\begin{proof}
It suffices to see that $i^M_E$ is continuous at $\delta$.
So let $f\in M$ and $\alpha<\delta$
with $f:\left<V_\alpha\right>^{<\om}\to\delta$,
and let
$a\in\left<V_{j(\alpha)}\right>^{<\om}$.
We want to see that $[a,f]^M_E<\delta$.
But $\cof(\delta)=\om$, so fix $g:\om\to\delta$ cofinal,
and for $n<\om$ let
\[ A_n=\{u\in\left<V_\alpha\right>^{<\om}\bigm|f(u)<g(n)\}.\]
Then $\left<V_\alpha\right>^{<\om}=\bigcup_{n<\om}A_n$,
so the usual argument gives $A_n\in 
E_a$ for some $n<\om$, which suffices.
\end{proof}

By the claim, $V_\delta\in\wfp(U)$
and $U\sats$``$V=L(V_\delta)$''.

\begin{clmtwo} $U$ is illfounded.\end{clmtwo}
\begin{proof}
 Suppose $U$ is wellfounded. Then note that 
$i^M_E(\alpha+\kappa)>\alpha+\kappa$,
and by the previous claim, that 
$\J_{i^M_E(\alpha+\kappa)+1}\sub U$,
so
\[ \pow(V_\delta)\inter\J_{\alpha+\kappa+2}\sub U.\]
But
$\pow(V_\delta)\inter U\sub M$,
because given any $A\in\pow(V_\delta)\inter U$,
we can find some pair $(a,f)$ such that $[a,f]^M_E=A$,
with $f\in M$ and $a\in V_\delta$, and since $E\in M$, it easily follows
that $A\in M$. Putting the $\sub$-statements together,
we contradict Fact \ref{fact:always_Hull_below_admissible}.
\end{proof}

By the above claim and Fact \ref{fact:wfp_admissible}, we have
$\J_\theta\sub U$, so $\pow(V_\delta)\inter\J_\theta\sub 
U$.
But then we reach a contradiction like in the proof of the claim.
This completes the proof in this case.

Now suppose instead that $\bar{\delta}<\delta$.
Fix a structure $M\sub V_{\bar{\delta}}$ which codes
$\J_{\bar{\theta}+1}(V_{\bar{\delta}})$,
where $\bar{\theta}=\kappa_{V_{\bar{\delta}}}$.
Then letting $U=\Ult_0(M,E_j)$ (computed in the codes)
and $k:M\to U$ be the ultrapower map,
we have $M,U,k\in\J_\theta(V_\delta)$.
Arguing as above, $k$ is continuous at $\bar{\delta}$,
so $\delta\in\wfp(U)$, and $U$ is illfounded,
because otherwise, by elementarity, it includes a code for $\J_\theta$
in it. But then by Fact \ref{fact:wfp_admissible}, $\J_\theta$``$\sub$''$U$
(that is, all elements of $\J_\theta$ are coded into $U$),
also a contradiction.
\end{proof}

We next observe that the  preceding result is optimal
in the case that $\bar{\delta}=\delta$
and we have a lot of $\AC$:\footnote{Thanks to Gabriel Goldberg for pointing
out the existence of \ref{fact:corazza}, and hence providing
the consistency upper bound for the hypothesis of 
\ref{tm:j_in_L(V_delta)_under_V=HOD}.}

\begin{fact}[Corazza]\label{fact:corazza}
 \tu{(}$\ZFC$\tu{)} Suppose $j\in\mathscr{E}(V_\delta)$,
 where $\delta\in\Lim$ \tu{(}so $\delta=\kappa_{\om}(j)$ 
and $V_\delta\sats\ZFC$\tu{)}. Then there is a 
set-forcing $\PP$
 which forces
(i) $V_\delta\sats\ZFC+$``$V=\HOD$'' and
(ii) there is $k\in\mathscr{E}(V_\delta)$
 with $\check{j}\sub k$.\footnote{In both statements here
 the $V_\delta$ is in the sense of the forcing extension.}
\end{fact}

\begin{tm}\label{tm:j_in_L(V_delta)_under_V=HOD}
\tu{(}$\ZF$\tu{)} Let $\delta\in\OR$ with
$V_\delta\sats\ZFC+$``$V=\HOD$'' and
$\mathscr{E}(V_\delta)\neq\emptyset$,
and  take $\delta$  least such.
Let $\theta=\kappa_{V_\delta}$ 
 and $M=\J_\theta(V_\delta)$.
 Then there is $j\in\mathscr{E}(V_\delta)$ which is
$\Sigma_1^M(\{V_\delta\})\wedge\Pi_1^M(\{V_\delta\})$,
meaning there are $\Sigma_1$ formulas $\varphi,\psi$ such that
\[ 
\all x,y\in V_\delta\ [j(x)=y\iff 
M\sats[\varphi(x,y,V_\delta)\wedge\neg\psi(x,y,
V_\delta)]].\]
\end{tm}
\begin{proof}
 Note $\cof(\delta)=\om$. The $j$ satisfying these requirements
 is just the left-most branch through the natural tree searching
 for such an embedding. 
 That is, let $T$ be the tree whose nodes are finite sequences
 \[ 
((j_0,\alpha_0,\beta_0),(j_1,\alpha_1,\beta_1),\ldots,
(j_{n-1},\alpha_{n-1},\beta_{n-1})) \]
 such that for each $i<n$,
 $j_i:V_{\alpha_i}\to V_{\beta_i}$ is elementary and $\kappa=\crit(j_i)$ 
exists,
 $V_\kappa\sats$``$V=\HOD$'',
 and if $i+1<n$ then 
$\beta_i<\alpha_{i+1}$
 and $j_i\sub j_{i+1}$.

 Now any $\Sigma_1$-elementary $j:V_\delta\to V_\delta$
 determines an infinite branch through $T$ (we can take $\alpha_0=\crit(j)$
 and $\beta_n=j(\alpha_n)$ and $\alpha_{n+1}=\beta_n+1$). 
This is clear
 enough except for the fact that $V_{\kappa}\sats$``$V=\HOD$''
 where $\kappa=\crit(j)$. But because $V_\delta\sats\ZFC$,
 we must have $\kappa_\om(j)=\delta$, and since 
$V_\delta\sats$``$V=\HOD$'', it follows that 
$V_\kappa\sats$``$V=\HOD$'' also. Conversely,
let $\left<(j_i,\alpha_i,\beta_i)\right>_{i<\om}$ be an infinite branch
through $T$, and  $\lambda=\bigcup_{i<\om}\alpha_i=\bigcup_{i<\om}\beta_i$.
Then $\lambda\in\Lim$ and $j\in\mathscr{E}_1(V_\lambda)$, and since 
$V_\kappa\sats$``$V=\HOD$''
where $\kappa=\crit(j)$, therefore $V_\lambda\sats\ZFC+$``$V=\HOD$''.
In fact $j$ is fully elementary (either since $V_\lambda\sats\ZF$,
or because $\cof(\lambda)=\om$ and by 
\cite[Theorem 5.6***]{cumulative_periodicity}),
so  $\lambda=\delta$
by  minimality.

Note that $T$ is definable over $V_\delta$.
Now the rank analysis of  $T$ is computed over $M$.
That is, given a node $t\in T$,
let
\[ T_t=\{s\in T\bigm| t\ins s\text{ or }s\ins t\}.\]
Then there is a rank function for $T_t$ (in $V$) 
iff there is one in $M$; this is a standard consequence of 
admissibility.
Let $<^*$ be the standard wellorder of $V_\delta$
resulting from the fact that $V_\delta\sats$``$V=\HOD$''.
Let $b=\left<t_i\right>_{i<\om}$ be the left-most branch
of $T$ with respect to $<^*$.
That is, $t_0=\emptyset$, $t_1=\left<(j_0,\alpha_0,\beta_0)\right>$
is the $<^*$-least node of $T$ of length $1$ such that
there is no rank function for $T_{t_1}$ (in $M$),
and then $t_2=\left<(j_i,\alpha_i,\beta_i)\right>_{i<2}$
is the $<^*$-least node of $T$ of length $2$,
extending $t_1$, such that there is no rank function for $T_{t_2}$ (in $M$),
etc. Note here that because $T_{t_n}$ has no rank function
(in $M$), $t_{n+1}$ does exist. This determines our branch $b$,
and hence a $\Sigma_1$-elementary $j:V_\delta\to V_\delta$.

Finally note that $b$ is appropriately definable.
\end{proof}

For our remaining results we also restrict to the case that 
$\bar{\delta}=\delta$. We will use the Cantor-Bendixson derivative analysis on 
the tree of attempts to build 
embeddings $j$,  in order to obtain further information about such embeddings, 
as done in 
\cite[\S4]{con_lambda_plus_2} (we repeat some common details here, though,
for self-containment):

\begin{dfn}
\tu{(}$\ZF$\tu{)} Let $\delta\in\Lim$ and $m\leq\om$.
 Then $T=T^{\delta,m}$ denotes the  following tree of  attempts to 
build
 a (possibly partial) $\Sigma_m$-elementary 
$j:_{\mathrm{p}}V_\delta\to V_\delta$.
 The nodes in $T$ are finite sequences
\[t=((j_0,\alpha_0,\beta_0),\ldots,(j_n,\alpha_n,\beta_n))\]
such that $j_i:V_{\alpha_i}\to V_{\beta_i}$ is $\Sigma_1$-elementary
and cofinal, $j_i:_{\mathrm{p}}V_\delta\to V_\delta$
is $\Sigma_m$-elementary on its domain $V_{\alpha_i}$,
$\beta_i<\alpha_{i+1}$,
and $j_{\alpha_i}\sub j_{\alpha_{i+1}}$. Write $j_t=j_n$.

For $\alpha\in\OR$, let $T_{\alpha}$ be the $\alpha$th derivative of $T$,
defined as follows.
Set $T_0=T$, and for limit $\lambda$ set 
$T_\lambda=\bigcap_{\alpha<\lambda}T_\alpha$.
Given $T_\alpha$, $T_{\alpha+1}$ is the set of all 
$t\in T_\alpha$ such that for every $\beta<\delta$
there is an extension $s$ of $t$ with $s\in T_\alpha$
such that $\beta\in\dom(j_s)$.
Let $T_\infty=T_{\OR}$. We say that $T_\infty$ is \emph{perfect}
iff for every $t\in T_\infty$
there are $s_1,s_2\in T_\infty$, both extending $t$,
such that $j_{s_1}\not\sub j_{s_2}\not\sub 
j_{s_1}$.

We write $[T]$ for the set of infinite branches through 
$T$. Clearly each $b\in[T]$ determines a $\Sigma_1$-elementary map 
$j_b:V_\lambda\to V_\lambda$
for some limit $\lambda\leq\delta$, and if $\lambda=\delta$
then $j$ is $\Sigma_m$-elementary.

We say that an embedding $k:V_\lambda\to V_\lambda$, for limit $\lambda$, is 
\emph{$V$-amenable}
iff $k\rest V_\alpha\in V$ for each $\alpha<\lambda$.
Clearly $j_b$ above is $V$-amenable.
\end{dfn}

Note $\left<T_{\alpha}\right>_{\alpha\in\OR}\in L(V_\delta)$,
 where $T=T^{\delta,m}$.
The following lemma shows that if $T_\infty\neq\emptyset$
then $T_\infty$ is, by a certain natural measure, of maximal complexity:
 
\begin{lem}\label{lem:T_infty_perfect}
\tu{(}$\ZF$\tu{)} Let $\delta\in\Lim$.
 Let $\gamma$ be least such that $T_\gamma=T_\infty$.
 Then:
 \begin{enumerate}
  \item If $\mathscr{E}_m(V_\delta)\neq\emptyset$
  then $T_\infty\neq\emptyset$.
 \item If $T_\infty\neq\emptyset$ then
  $T_\infty$ is perfect and $\gamma=\kappa_{V_{\delta}}$.
\item If $T_\infty=\emptyset$ then $\gamma<\kappa_{V_\delta}$.
 \end{enumerate}
\end{lem}
\begin{proof}
If $j\in\mathscr{E}_m(V_\delta)$ then easily $T_\infty\neq\emptyset$,
and in fact, if we force over $V$ to collapse $V_\delta$
to become countable, then in $V[G]$,
there is an infinite branch $b\in[T_\infty]$ with $j_b=j$.

 The fact that $\gamma\leq\kappa=\kappa_{V_\delta}$ is standard:
 Suppose not and let $t\in T_\kappa\cut T_{\kappa+1}$.
 Then we can fix $\alpha<\delta$ such that no $s\in T_\kappa$
 extending $t$ has $\alpha\in\dom(j_s)$.
 But then $\J_\kappa(V_\delta)\sats$``For every $s\in T$
 extending $t$ with $\alpha\in\dom(j_s)$
 there is $\beta\in\OR$ such that $s\notin T_\beta$''.
 By admissibility, it follows that there is $\xi<\kappa$
 such that $\J_\xi(V_\delta)$ satisfies this.
 But then note that $t\notin T_\kappa$, contradiction.
 The same argument (but slightly simpler) shows that if $T_\infty=\emptyset$
 then $\gamma<\kappa$.
 
 Now suppose that $T_\infty\neq\emptyset$
 but $T_\infty$ is not perfect.
 Then fix $t\in T_\infty$
 such that for all $s_1,s_2\in T_\infty$
 extending $t$, we have $j_{s_1}\sub j_{s_2}$
 or $j_{s_2}\sub j_{s_1}$.
 Let $\mathscr{S}$ be the set of all $s\in T_\infty$ extending $t$,
 and $j=\bigcup_{s\in\mathscr{S}}j_s$.
 Then note that $j:V_\delta\to V_\delta$ is a well-defined function
 and is $\Sigma_m$-elementary,
 and $j\in L(V_\delta)$.
 
 In fact, $j\in\J_\kappa$,
 contradicting Theorem \ref{tm:no_j_in_adm}.
 For fix $\alpha<\delta$ with $\alpha>\dom(j_t)$; so $j\rest V_\alpha\in 
V_\delta$.
Then for all nodes $s\in T$ extending $t$ with $V_\alpha\sub\dom(j_s)$
and $j\rest V_\alpha\neq j_s\rest V_\alpha$,
there is $\xi<\kappa$ with $s\notin T_\xi$.
So by admissibility, there is $\xi<\kappa$ such that $s\notin T_\xi$
for all such $s$. Therefore, for each $\alpha<\delta$
such that $\alpha>\dom(j_t)$,
there is $\beta<\delta$ and a map $k:V_\alpha\to V_\beta$
(actually $k=j\rest V_\alpha$) with $j_t\sub k$
and there is $\xi<\kappa$ such that $s\notin T_\xi$
for all $s$ as above. So by admissibility, there is $\xi<\kappa$
which works simultaneously for all $\alpha<\delta$.
But then clearly $j$ is definable from parameters over $\J_\xi$,
so $j\in\J_\kappa$, contradicting Theorem \ref{tm:no_j_in_adm}.

It remains to see that if $T_\infty\neq\emptyset$, hence perfect,
then $\gamma=\kappa_{V_\delta}$.
So suppose otherwise.
Let $\gamma<\xi<\kappa_{V_\delta}$.
Then $T_\infty\in\J_\xi$, and we can force over $L(V_\delta)$
with $T_\infty$, in the obvious manner, with the generic filter $G$ being an 
infinite branch $b$ through $T_\infty$, and note that by genericity,
$j_b:V_\delta\to V_\delta$ is $\Sigma_m$-elementary (that is, genericity
ensures that $\dom(j_b)=V_\delta$). Note that $j_b$ is $V$-amenable.

Now we will proceed through basically the argument from before,
but just need to see that things adapt alright to the generic embedding $j_b$.
We first consider the finite iterates $(j_b)_n$ of $j_b$ and the eventual 
fixedness of ordinals (that is, whether $(j_b)_n(\alpha)=\alpha$ for some 
$n$).
If $\alpha<\delta$ is a limit and $j_b(\alpha)=\alpha$
then $j_b\rest V_\alpha$ determines
$(j_b)_n\rest V_\alpha$ for each $n<\om$ as usual,
and this is all in $V$, so all ordinals $<\alpha$ are eventually 
fixed. So we may assume that there is a bound $<\delta$
on such ordinals $\alpha$. If $\delta=\alpha+\om$ 
for some limit $\alpha$
then clearly $j_b(\alpha+n)=\alpha+n$ for all $n<\om$,
so we are also done in this case. So we are left with the case
that $\delta$ is a limit of limits, and there is $\alpha<\delta$ such that 
$j_b$ fixes no ordinal
in $[\alpha,\delta)$, and take $\alpha$ least such. In particular,
$j_b(\alpha)>\alpha$. Letting 
$\alpha_0=\alpha$ and $\alpha_{n+1}=j_b(\alpha_n)$,
note that $\sup_{n<\om}\alpha_n=\delta$ (for if 
$\eta=\sup_{n<\om}\alpha_n<\delta$
then $j_b\rest V_\eta\in V$, so $\left<\alpha_n\right>_{n<\om}\in 
V$, so $\cof^V(\eta)=\om$, so $j_b(\eta)=\eta$, contradiction).\footnote{Since 
$j_b\notin V$, it seems that $\delta$ might not have cofinality $\om$ in $V$ 
here.}
But $j_b$ has unboundedly many fixed points ${<\alpha_0}$.
Therefore $(j_b)_1$ has unboundedly many ${<\alpha_1}$,
etc, $(j_b)_{n}$ has unboundedly many ${<\alpha_n}$.
But then $(j_b)_{n}\rest V_{\alpha_n}$ is enough to determine
$(j_b)_{n+k}\rest V_{\alpha_n}$ for $k<\om$
(working in $V$), so we get that all points $<\alpha_n$ are eventually fixed.
So if $\delta$ is singular in $V$, we can find $n<\om$
such that $(j_b)_n(\cof^V(\delta))=\cof^V(\delta)$.
Similarly, if $\delta$ is regular but non-inaccessible,
we can find $n<\om$ such that $(j_b)_n(\scot^V(\delta))=\scot^V(\delta)$
(for this, argue as before to first find $n<\om$
and a limit $\alpha<\delta$ such that $(j_b)_n$
has cofinally many fixed points ${<\alpha}$
and $\scot^V(\delta)\in V_\alpha$, and then proceed in $V$).

So fix $n<\om$ such that $k=(j_b)_n$ is like this.
Then for each $\eta\in\OR$, $\Sigma_0$-{\L}o\'{s}' theorem holds for 
$\Ult_0(\J_\eta,E_{k})$
The proof is just like before -- the fact that $k\notin V$ does not matter.
(The ultrapower is formed using only functions in $V$,
so all the calculations with partitioning measure one sets is done in $V$,
and because either $V_\delta$ is inaccessible in $V$ or $k$ fixes the relevant 
objects, the argument goes through.)

So consider $U=\Ult_0(\J_{\chi},E_k)$
where $\chi=(\om\dot(\delta+\xi))+\crit(k)+1$ (we had 
$\PP=T_\infty\in\J_\xi$).
We claim that $U$ is illfounded.
For otherwise $\OR^U>\chi$,
so as before, we get 
\[t=\Th_{\Sigma_1}^{\J_\chi}(V_\delta\cup\{V_\delta\})
\in\J_{\chi}[G].\]
Let $\tau\in\J_{\chi}$ be a $T_\infty$-name
such that $\tau_G=t$.
Now $t\in L(V_\delta)$, and since $G$ is $L(V_\delta)$-generic,
there is $p\in G$ such that $L(V_\delta)\sats$``$p$ forces 
$\check{t}=\tau$''. But then for $(\varphi,x)\in V_\delta$ we have
\[ (\varphi,x)\in t\iff
\J_\chi\sats p\forces_{\PP}(\varphi,x)\in\tau, \]
because for example if $(\varphi,x)\in t$ but there is $q\leq p$  and 
$\J_\chi\sats q\forces_{\PP}(\varphi,x)\notin\tau$,
then $p$ cannot have forced $\check{t}=\tau$;
moreover here this forcing relation is definable over $\J_\chi$,
and in fact, it is definable from parameters over $\J_{\chi-1}$.
For the $\Sigma_0$ forcing relation over $\J_{\om\cdot(\delta+\xi)}$
is $\Delta_1^{\J_{\om\cdot(\delta+\xi)}}(\{\PP\})$,
because we have enough closure at this stage,
and this is then maintained level by level, and the $\Sigma_0$ forcing relation 
over $\J_{\chi-1}$
is $\Delta_1^{\J_{\chi-1}}(\{\PP\})$,
and the $\Sigma_n$ forcing relation for $\bfSigma_n^{\J_{\chi-1}}$-definable 
names $\sub\J_{\chi-1}$ is definable from $\PP$ over $\J_{\chi-1}$,
but $\tau$ can be taken to be such a name. Hence we get $t\in\J_\chi$,
which is a contradiction.

So $U$ is illfounded, and hence $\J_{\kappa_{V_\delta}}\sub\wfp(U)$.
But then we again get $t\in\J_\chi[G]$, which is again a contradiction,
completing the proof.
\end{proof}

\begin{tm}\label{tm:DC_or_force_embeddings} \tu{(}$\ZF$\tu{)} Let 
$\delta\in\Lim$
with $\mathscr{E}(V_\delta)\neq\emptyset$,
 where $m\in[1,\om]$.
 Then:
 \begin{enumerate}
  \item In a set-forcing extension of $V$, for each $V$-amenable
$j\in\mathscr{E}_m(V_\delta)$ and $\alpha<\delta$
 there is a $V$-amenable $k\in\mathscr{E}_m(V_\delta)$
 with $k\rest V_\alpha=j\rest V_\alpha$
 but $k\neq j$.

  \item If $\DC$ holds and $\cof(\delta)=\om$ then for each 
$j\in\mathscr{E}_m(V_\delta)$ and $\alpha<\delta$ there 
is 
 $k\in\mathscr{E}_m(V_\delta)$ with $k\rest V_\alpha=j\rest 
V_\alpha$
 but $k\neq j$.
 \end{enumerate}
\end{tm}
\begin{proof}
By Lemma \ref{lem:T_infty_perfect}, $T_\infty$ is perfect (notation as 
there),  which immediately gives the theorem (of course we
can take the generic extension to be $V[G]$ where $G$ collapses $V_\delta$ to 
become countable).
\end{proof}

We now show that the kind of embedding defined in Theorem 
\ref{tm:j_in_L(V_delta)_under_V=HOD}  cannot be extended to 
the whole of $\J_\theta(V_\delta)$:
\begin{tm}\tu{(}$\ZF$\tu{)}
 Let $\delta\in\Lim$ and $j\in\mathscr{E}(V_\delta)$. Let 
$\theta=\kappa_{V_\delta}$.
 Suppose that $j$ is definable
 over $\J_\theta(V_\delta)$. Then there is 
$\alpha<\theta$
 such that $\Ult_0(\J_\alpha(V_\delta),E_j)$
 is illfounded.
\end{tm}
\begin{proof}
Write $\J_\alpha$ for $\J_\alpha(V_\delta)$.
Since $j\in L(V_\delta)$, we have
$\cof(\delta)=\cof^{L(V_\delta)}(\delta)=\om$.

\begin{clm}Let $\alpha<\theta$ with $\alpha$ either a successor or
$\cof(\alpha)<\delta$ and
 $j$ continuous at $\cof(\alpha)$.
Let $f:\left<V_\delta\right>^{<\om}\to\J_\alpha$
and $a\in\left<V_\delta\right>^{<\om}$.
Then there is $g\in\J_\alpha$ such that $g(u)=f(u)$ for $E_a$-measure one many 
$u$.\end{clm}
\begin{proof} By the usual calculations using the continuity
of $j$, we can assume that $\rg(f)\sub x$ for some $x\in\J_\alpha$.
But by Fact \ref{fact:always_Hull_below_admissible}, there is a surjection 
$\pi:V_\delta\to x$ with $\pi\in\J_\alpha$. Therefore, we can assume that 
$x=V_\delta$.
(That is, given $y\in x$, let $Z_y=\{w\in V_\delta\bigm|\pi(w)=y\}$,
and then let $z_y=Z_y\inter V_\beta$ where $\beta$ is least such that this 
intersection is $\neq\emptyset$, and then let 
$f':\left<V_\delta\right>^{<\om}\to V_\delta$ be $f'(u)=z_{f(u)}$. Then
clearly $f'$ codes $f$ modulo $\pi$.) But since $\cof(\delta)=\om$,
there is then $\beta<\delta$ such that $f(u)\in V_\beta$
for $E_a$-measure one many $u$. But then restricting to this set,
we get a function $g\in V_\delta$.\end{proof}

So let
$M=\J_\theta$ and $U=\Ult_0(M,E)$.
Suppose now that the theorem fails.

\begin{clm}$U=M$ is wellfounded and $j_+=i^M_E:M\to U$ is cofinal and 
$\Sigma_1$-elementary, with $j_+(V_\delta)=V_\delta$ and $j\sub j_+$.
\end{clm}
\begin{proof}
We also have $\Sigma_0$-{\L}o\'{s}' criterion
for the ultrapower, by Lemma \ref{lem:V=L(V_delta)_cof(delta)=om}. This gives 
the 
$\Sigma_1$-elementarity of $i^M_E$.
And note that by the previous claim and our contradictory hypothesis,
$U$ is wellfounded.  The fact that $j_+(V_\delta)=V_\delta$
follows from the continuity of $j_+$ at $\delta$, which
holds because 
$\cof(\delta)=\om$, like in the proof of the previous claim.

It remains to see that $U=M$. Note that $U=\J_\gamma(V_\delta)$
for some $\gamma$, and certainly $\theta\leq\gamma$.
But 
$M\sats$``There is no $\alpha>\delta\in\OR$ such that $\J_\alpha$ is 
admissible'', so by $\Sigma_1$-elementarity
and as $j_+(V_\delta)=V_\delta$,  $U$ satisfies this 
statement. Therefore $\gamma\leq\theta$, so we are done.
\end{proof}

\begin{clm}Given any $\Sigma_1$-elementary $k:V_\delta\to V_\delta$,
 there is at most one extension of $k$ to a $\Sigma_1$-elementary
 $k_+:\J_\theta\to\J_\theta$ with $k_+(V_\delta)=V_\delta$.
 Moreover, if $k$ is definable from parameters
 over $\J_\theta$ then so is $k_+$.
\end{clm}
\begin{proof}
Let us first observe that $k_+\rest\theta$ is uniquely determined.
Given $\alpha<\theta$, there is a $\Sigma_1$ formula $\varphi$
and $z\in V_\delta$ such that $\alpha$ is the least $\alpha'\in\OR$
such that $\J_{\alpha'}\sats\psi(z,V_\delta)$, where
\[\psi(\dot{z},\dot{v}) =\text{``}
\all x\in\dot{v}\ \exists y\ \varphi(x,y,\dot{z},\dot{v})\text{''}.\]
Hence $k_+(\alpha)$ must be the least $\alpha'\in\OR$
such that $\J_{\alpha'}\sats\psi(k(z),V_\delta)$.

But now $\J_\alpha=\Hull_1^{\J_\alpha}(V_\delta\cup\{V_\delta\})$,
and since $k_+\rest\J_\alpha:\J_\alpha\to\J_{k_+(\alpha)}$
must be $\Sigma_1$-elementary and have $k_+(V_\delta)=V_\delta$,
$k$ determines 
$k_+\rest\J_\alpha$.

The ``moreover'' clause clearly follows from the manner in which
we have computed $k_+$ above from $k$.\end{proof}

\begin{clm}Let $k_+:\J_\theta\to\J_\theta$ be $\Sigma_1$-elementary
 with $k_+(V_\delta)=V_\delta$.
 Then $k_+$ is fully elementary.
\end{clm}
\begin{proof}
Given $\gamma\leq\delta$ and $\alpha\leq\theta$,
let $t_\gamma^\alpha=\Th_{\Sigma_1}^{\J_\alpha}(V_\gamma\cup\{V_\delta\})$.
Let $\widetilde{t}_\gamma^\alpha$ code $t_\gamma^\alpha$
as a subset of $V_\gamma$.
We claim that $k_+(\widetilde{t}_\delta^\theta)=\widetilde{t}_\delta^\theta$.
For given $\gamma<\delta$, since $\widetilde{t}^\theta_\gamma\in V_\delta$,
by admissibility there is $\alpha_\gamma<\theta$ such that
$t^\beta_\gamma=t^\theta_\gamma$ for all 
$\beta\in[\alpha_\gamma,\theta]$.
But then it easily follows that
$k(\widetilde{t}^\theta_\gamma)=\widetilde{t}^\theta_{k(\gamma)}$.
But then it follows that 
$k(\widetilde{t}^\theta_\delta)=\widetilde{t}^\theta_\delta$.

Also, if $\delta$ is singular in $\J_\theta$
then $k$ is continuous at $\cof^{\J_\theta}(\delta)$,
because $k(\delta)=\delta$. Similarly if $\J_\theta\sats$``$\delta$ is regular 
but not inaccessible''. 

From here we can argue as in the proof of 
\cite[Theorem 5.6***]{cumulative_periodicity}.
\end{proof}

Using the preceding claims,
we can now derive the usual kind of contradiction,
 considering the least critical point $\kappa$
 of any  $\Sigma_1$-elementary $k_+:\J_\theta\to\J_\theta$ such that 
$k_+(\delta)=\delta$ and $k_+$ is $\Sigma_n$-definable from parameters
over $\J_\theta$ (for some appropriate $n$).
This completes the proof.
\end{proof}

\begin{rem}
The argument above shows that
 if $\delta$ is a limit and $\kappa=\kappa_{V_\delta}$,
 then $\cof^{L(V_\delta)}(\kappa)=\cof^{L(V_\delta)}(\delta)$,
 definably over $\J_\kappa(V_\delta)$. It also shows that,
 with $T_\infty$ as before, for each $\alpha<\delta$,
 since $T_\infty\inter V_\alpha\in V_\delta$,
 there is $\gamma<\kappa$ such that $T_\gamma\inter V_\alpha=T_\infty\inter 
V_\alpha$. Also note that, for example, if $\delta=\lambda+\om$,
then $T_\infty$ 
cannot be just finitely splitting
beyond some node $t$ (counting here the number of nodes
beyond $t$ in each $V_{\lambda+n}$, for $n<\om$).
For otherwise (taking $t$ with $\dom(j_t)=V_{\lambda+n}$ for some $n$),
the 
tree structure
of $(T_\infty)_t$ is coded by a real, hence is in $V_\delta$, and an 
admissibility argument
easily gives that $(T_\infty)_t$ is computed by some $\gamma<\kappa$,
which leads to a contradiction as before. However, such an argument
doesn't seem to work in the case that $(T_\infty)_t$ is just
$\om$-splitting.
\end{rem}

\begin{rem}\label{rem:j_in_L(V_delta)}
In the only example we have where $\delta\in\Lim$ and
and $j\in\mathscr{E}(V_\delta)\inter L(V_\delta)$,
we have $V_\delta\sats\ZFC+$``$V=\HOD$'',
so $L(V_\delta)\sats\AC$ and $\delta=\kappa_\om(j)$.
Is it possible to have $j\in L(V_\delta)\inter\mathscr{E}(V_\delta)$
with $\kappa_\om(j)<\delta$? Or even with $\kappa_\om(j)=\delta$ but 
$L(V_\delta)\sats\neg\AC$? Is it possible to have this with 
$(V_\delta,j)\sats\ZFR$? We know that we need $\cof^{L(V_\delta)}=\om$
for this. We have the tree $T_\infty\in L(V_\delta)$.
But without left-most branches, it is not clear to the author
how to get an embedding in $L(V_\delta)$ from this.
Relatedly, is it possible for $j$ to be generic over $L(V_\delta)$
and have $V_\delta=V_\delta^{L(V_\delta)[j]}$?
\end{rem}

\bibliographystyle{plain}
\bibliography{../bibliography/bibliography}

\end{document}